\documentclass{amsart}
 \usepackage{amsmath}
\usepackage{amsfonts}
\usepackage{amssymb}
\usepackage{url}
\usepackage{epsfig}

\usepackage{epstopdf}
\usepackage{color}
\usepackage{bbm}

\definecolor{blue}{RGB}{197,77,87}
\definecolor{green}{RGB}{200,244,99}
\definecolor{red}{RGB}{78,205,196}

\definecolor{grey}{RGB}{197,77,87}

\newcommand{\blue}[1]{\color{blue}#1}
\newcommand{\green}[1]{\color{green}#1}
\newcommand{\red}[1]{\color{red}#1}
\newcommand{\grey}[1]{\color{grey}#1}

 \newtheorem{thm}{Theorem}[section]
   \newtheorem{lem}[thm]{Lemma}

    \newtheorem{cor}[thm]{Corollary}
   \newtheorem{pro}[thm]{Proposition}
 
 \newtheorem{alg}[thm]{Algorithm}
   \theoremstyle{definition}
   \newtheorem{defn}[thm]{Definition}
   \newtheorem{ex}[thm]{Example}
    \newtheorem{rem}[thm]{Remark}

\newcommand{\CC}{\mathbb C}
\newcommand{\RR}{\mathbb R}
\newcommand{\QQ}{\mathbb Q}
\newcommand{\ZZ}{\mathbb Z}

\newcommand{\TT}{\mathbb{T}}

\newcommand{\LL}{\mathbb L}
\newcommand{\cT}{\mathcal T}
\newcommand{\cF}{\mathcal F}
\newcommand{\cA}{\mathcal A}
\newcommand{\cB}{\mathcal B}
\newcommand{\cR}{\mathcal R}
\newcommand{\cU}{\mathcal U}

\newcommand{\cM}{\mathcal M}
\newcommand{\cN}{\mathcal N}
\newcommand{\kk}{\mathbf k}

\newcommand{\cprod}{\prod}
\newcommand{\stint}{\cap_{\text{st}}}

\DeclareMathOperator{\supp}{supp}
\DeclareMathOperator{\conv}{conv}

\DeclareMathOperator{\vol}{vol}
\DeclareMathOperator{\val}{val}
\DeclareMathOperator{\init}{in}
\DeclareMathOperator{\Cay}{Cay}

\DeclareMathOperator{\row}{rowspace}
\DeclareMathOperator{\link}{link}
\DeclareMathOperator{\mult}{mult}
\DeclareMathOperator{\rank}{rank}

\DeclareMathOperator{\spann}{span}
\DeclareMathOperator{\codim}{codim}

\begin{document}

 \title{Computing Tropical Resultants}
 \author{Anders Jensen}
 \author{Josephine Yu}
 \address{Institut for Matematik, Aarhus Universitet, Aarhus, Denmark}
\email{jensen@imf.au.dk}
\address{School of Mathematics, Georgia Institute of Technology,
        Atlanta GA, USA}
\email {jyu@math.gatech.edu}


\date{\today}

 \maketitle
 
 \begin{abstract}
We fix the supports $\cA=(A_1,\dots,A_k)$ of a list of tropical polynomials and define the tropical resultant $\cT\cR(\cA)$ to be the set of choices of coefficients such that the tropical polynomials have a common solution. We prove that $\cT\cR(\cA)$ is the tropicalization of the algebraic variety of solvable systems and that its dimension can be computed in polynomial time. The tropical resultant inherits a fan structure from the secondary fan of the Cayley configuration of~$\cA$, and we present algorithms for the traversal of $\cT\cR(\cA)$ in this structure. We also present a new algorithm for recovering a Newton polytope from the support of its tropical hypersurface. We use this to compute the Newton polytope of the sparse resultant polynomial in the case when $\cT\cR(\cA)$ is of codimension 1. Finally we consider the more general setting of specialized tropical resultants and report on experiments with our implementations. 
 \end{abstract}



\section{Introduction}

We study generalizations of the problem of computing the Newton polytope of the sparse resultant combinatorially, without first computing the resultant polynomial.  The input is a tuple $\cA = (A_1, A_2, \dots, A_k)$ of integer point configurations in $\ZZ^n$.  The {\em sparse resultant} $\cR(\cA)$ of $\cA$, or the \emph{variety of solvable systems}, is the closure in $(\CC^*)^{A_1}\times(\CC^*)^{A_2}\times\cdots\times(\CC^*)^{A_k}$ of the collection of tuples of polynomials $(f_1, f_2, \dots, f_k)$ such that $f_1=f_2=\cdots=f_k=0$ has a solution in $(\CC^*)^n$ and each $f_i$ has support $A_i$.  This variety is irreducible and defined over $\QQ$ \cite{Sturmfels94}. 
If $\cR(\cA)$ is a hypersurface, then it is defined by a polynomial, unique up to scalar multiple, called the (sparse) \emph{resultant polynomial} of $\cA$. Its Newton polytope is called the {\em resultant polytope} of $\cA$.

In the hypersurface case, Sturmfels gave a combinatorial description of the resultant polytope \cite{Sturmfels94}, giving rise to a combinatorial algorithm for computing its vertices from the vertices of the secondary polytope of the Cayley configuration $\Cay(\cA)$.  A drawback of this construction is that the secondary polytope typically has far more vertices than the resultant polytope.  There have been attempts to compute the resultant polytopes without enumerating all vertices of the secondary polytope \cite{EFK}.  A main contribution of our paper is an algorithm (Section~\ref{sec:algorithms}) for traversing the tropicalization of $\cR(\cA)$ as a subfan of the secondary fan of $\Cay(\cA)$.  This approach allows us to compute tropicalizations of resultant varieties of arbitrary codimension. 

The {\em tropical resultant} $\cT\cR(\cA)$ consists of tuples of tropical polynomials having a common solution.  We show in Theorem~\ref{thm:tropicalequalstropicalized} that $\cT\cR(\cA)$ coincides with the tropicalization of $\cR(\cA)$.  The tropical resultant is combinatorial in nature, and we present in Theorem~\ref{thm:orthants} a simple description of it as a union of polyhedral cones, each of which is the sum of a positive orthant and a linear space. 

In \cite{DFS}, the tropical discriminant is described as a sum of a tropical linear space and an ordinary linear space.  This description carries over to the tropical resultant when $\cA$ is {\em essential}, and in particular $\cR(\cA)$ is a hypersurface.  Our description in Theorem~\ref{thm:orthants} is different and also works for non-essential cases and non-hypersurface cases.  Moreover, it is simpler, and we do not need to compute a nontrivial tropical linear space.

The tropicalization of a variety is a polyhedral fan of the same dimension as the original variety.  We derive a new formula for the codimension of the (tropical) resultant in Theorem~\ref{thm:codimension} and show that it can be computed in polynomial time using the cardinality matroid intersection algorithm.

Specialized resultants are obtained by fixing some coefficients of $f_i$'s and considering the collection of other coefficients giving a polynomial system solvable in the algebraic torus.  In other words, the specialized resultants are intersections of sparse resultants and subspaces parallel to coordinate subspaces.  When the specialized coefficient values are generic, the tropicalization $\cT\cR_S(\cA)$ of the specialized resultant is the stable intersection of the tropical resultant $\cT\cR(\cA)$ with a coordinate subspace. This is a subfan of the restriction of the secondary fan of $\Cay(\cA)$ to the subspace and can be computed by a fan traversal. The algorithms are significantly more complex and are described in Section~\ref{sec:special}.  Moreover, using the results from our concurrent work on tropical stable intersections \cite{stableIntersection}, we describe the specialized tropical resultant as a union of cones, each of which is the intersection of a coordinate subspace and the sum of a positive orthant and a linear space.

Computation of resultants and specialized resultants, of which the implicitization problem is a special case, is a classical problem in commutative algebra that remains an active area.  In the concurrent work \cite{EFKP} an algorithm for computing Newton polytopes of specialized resultant polynomials using Sturmfels' formula and the beneath-beyond method is presented and implemented, and the work is therefore highly relevant for our project.  
While the main focus of \cite{EFKP} is the efficiency of the computation of the Newton polytopes of specialized resultant polynomials, our main interest has been the geometric structure of secondary fans which allows traversal of tropical resultants of arbitrary codimension.

The {\em tropical description} of a polytope $P$ is a collection of cones whose union is the support of the codimension one skeleton of the normal fan of $P$, with multiplicities carrying lengths of the edges of $P$.  That is, the union is the {\em tropical hypersurface} defined by $P$. For example, the tropical hypersurface of a zonotope is the union of the dual hyperplanes (zones), and the tropical hypersurface of the secondary polytope of a point configuration contains codimension one cones spanned by vectors in the Gale dual.  See Section~\ref{sec:tropSecondary}.  
The tropical description uniquely identifies the polytope up to translation, and we consider it to be an equally important representation of a polytope as the V- and H-descriptions. Furthermore, the conversion algorithms between these representations deserve the same attention as other fundamental problems in convex geometry. 
A contribution of this paper is an algorithm (Algorithm~\ref{alg:region}) for reconstructing normal fans of polytopes from their tropical descriptions. We apply the algorithm to the tropical description of resultant polytopes in Theorem~\ref{thm:orthants} to recover the combinatorics of the resultant polytope.  From the normal fan, we can efficiently obtain the V-description of the polytope.

All the algorithms described in this paper have been implemented in the software Gfan \cite{gfan}. Computational experiments and examples are presented in Section~\ref{sec:comparison}.  A list of open problems is presented in Section~\ref{sec:openProblems}.

\section{Resultants}

Let $\cA = (A_1, A_2, \dots, A_k)$ where each $A_i = \{ a_{i,1}, a_{i,2}, \dots, a_{i,m_i} \}$ is a 
multi-subset of $\ZZ^n$, and let $m = m_1 + m_2 + \cdots + m_k$.  Throughout this paper, we assume that $m_i \geq 2$ for all $i$.  However, the points in $A_i$ need not be distinct.  This is important for some applications such as implicitization.  Let $Q_1, Q_2, \dots, Q_k$ be the convex hulls of $A_1, A_2, \dots, A_k$ respectively.  Let $(\CC^*)^{A_i}$ denote the set of polynomials of the form
$\sum_{j = 1}^{m_i} c_j x^{a_{ij}}$ in $\CC[x_1, x_2, \dots, x_n]$, where each $c_j$ is in $\CC^*  := \CC \backslash \{0\}$.
Let $Z \subseteq \cprod_{i = 1}^k (\CC^*)^{A_i}$ be the set consisting of tuples $(f_1, f_2, \dots, f_k)$ such that the system of equations $f_1 = f_2 = \cdots = f_k = 0$ has a solution in $(\CC^*)^n$.

\begin{defn}
The {\em resultant variety}, or the {\em variety of solvable systems}, is the closure $\overline{Z}$ of $Z$ in $\cprod_{i=1}^k (\CC^*)^{A_i}$ and is denoted $\cR(\cA)$.
\end{defn}

The resultant variety is usually defined as a subvariety of $\cprod_{i=1}^k \CC^{A_i}$ or its projectivization \cite{GKZ, Sturmfels94}, but we chose to work in $\cprod_{i=1}^k (\CC^*)^{A_i}$ as tropicalizations are most naturally defined for subvarieties of tori.

\subsection{A simple description of the tropical resultant and its multiplicities}
\label{sec:simple}

The tropical semiring $\TT = (\RR,\oplus, \odot)$ is the set of real numbers with minimum as tropical addition $\oplus$ and usual addition as tropical multiplication $\odot$.  A tropical (Laurent) polynomial $F$ in $n$ variables $x = (x_1, x_2, \dots, x_n)$ is a multiset of terms $(c, a)$ or $c \odot x^a$ where $c\in \RR$ is the coefficient and $a = (a_1, a_2, \dots, a_n) \in \ZZ^n$ is the exponent.  We will also write $F = \bigoplus_{(c,a) \in F} (c\odot x^a)$.  The \emph{support} of $F$ is the multiset of $a$'s, and the \emph{Newton polytope} of $F$ is the convex hull of its support.

The {\em tropical solution set} $\cT(F)$ of a tropical polynomial $F$ is the locus of points $x \in \RR^n$ such the minimum is attained at least twice in the expression
$$\bigoplus_{(c,a) \in F} (c\odot x^a) = \min_{(c, a) \in F}(c+a_1 x_1+ a_2 x_2+ \cdots + a_n x_n).$$  
In other words, a point $x \in \RR^n$ is in $\cT(F)$ if and only if the minimum for $(1,x)\cdot (c,a)$ is attained for two terms in $F$, which may be repeated elements.
Therefore, $\cT(F)$ is a (not necessarily pure dimensional) subcomplex of a polyhedral complex dual to the marked regular subdivision of the support of $F$ induced by the coefficients $c$, consisting of duals of cells with at least two marked points.  See Section~\ref{sec:secondary} for definitions of subdivisions and marked points.

When $F$ contains no repeated elements, the tropical solution set coincides with the non-smooth locus of the piecewise-linear function from $\RR^n$ to $\RR$ given by $x \mapsto F(x) = \bigoplus_{(c,a) \in F} (c\odot x^a)$, which is also called a {\em tropical hypersurface}.  In particular, if all coefficients of $F$ are the same and if $F$ contains no repeated elements, then the tropical hypersurface is the codimension one skeleton of the inner normal fan of the Newton polytope of $F$.

 Let $\cA = (A_1, A_2, \dots, A_k)$ be as before, and let $\RR^{A_i}$ denote the set of tropical polynomials of the form $\bigoplus_{j = 1}^{m_i} \left( c_{ij} \odot x^{\odot a_{ij}} \right)$. 

\begin{defn}
The {\em tropical resultant} $\cT\cR(\cA)$ of $\cA$ is the subset of $\RR^m$, or $ \RR^{A_1} \times \RR^{A_2} \times \cdots \times \RR^{A_k}$, consisting of tuples $(F_1, F_2, \dots, F_k)$ such that the tropical solution sets of $F_1, F_2, \dots, F_k$
 have a nonempty common intersection in $\RR^n$. 
\end{defn}

We can also consider the tropical resultant as a subset of $\cprod_{i=1}^k \RR^{A_i} / (1,1,\dots,1)\RR$, but we prefer to work with $\RR^m$ in this paper.  

For two univariate tropical polynomials, the term ``tropical resultant'' had been used by other authors to describe a tropical polynomial analogous to ordinary resultants.  In \cite{Odagiri} it is defined as the tropical determinant of the tropical Sylvester matrix.  In \cite{Tabera} it is defined as the tropicalization of the ordinary resultant polynomial.  In this paper the term ``tropical resultant'' always refers to a fan and never a tropical polynomial.  

\begin{defn}
Let $\kk$ be a field and $I\subseteq\kk[x_1,\dots,x_n]$ an ideal. The {\em tropical variety} $\cT(I)$ of $I$, or the \emph{tropicalization} of $V(I)$, is a polyhedral fan with support
$$
\cT(I) := \{\omega \in \RR^n: \text{ the initial ideal }\init_\omega(I) \text{ contains no monomials} \}.
$$
For $\omega$ in the relative interior of a cone $C_\omega \in \cT(I)$ we define its multiplicity as
$$
\mult_\omega(\cT(I)) := \dim_\kk(\kk[\ZZ^n \cap C_\omega^\perp] / \langle\init_\omega(I)\rangle)
$$
when the right hand side is finite, in particular when $C_\omega$ is a Gr\"obner cone of the same dimension as $\cT(I)$.
\end{defn}

In this definition we refer to the ``constant coefficient'' initial ideal as in \cite{BJSST}, where we disregard any valuation of the ground field even if it is non-trivial, except that we are picking out the terms with smallest $\omega$-degree. If the ideal $I$ is homogeneous, $\cT(I)$ gets a fan structure from the Gr\"obner fan  of $I$.
When $C_\omega$ is the smallest Gr\"obner cone in $\cT(I)$ containing $\omega$, the initial ideal $\init_\omega(I)$ is homogeneous with respect to any weight in the linear span of $C_\omega$.  Hence after multiplying each homogeneous element of $\init_\omega(I)$ by a Laurent monomial they generate an ideal $\langle\init_\omega(I)\rangle$ in the ring $\kk[\ZZ^n \cap C_\omega^\perp]$ of Laurent polynomials in $x_1, x_2, \dots, x_n$ which are of degree zero with respect to the weight vector $\omega$.


The following is the first main result toward a combinatorial description of the tropicalization of $\cR(\cA)$.  

\begin{thm}
\label{thm:tropicalequalstropicalized}
The support of the tropicalization of the resultant variety $\cR(\cA)$ coincides with the tropical resultant $\cT\cR(\cA)$.
\end{thm}
A consequence is that we may identify $\cT\cR(\cA)$ with the tropicalization of $\cR(\cA)$ and we define its multiplicities accordingly.

We will use incidence varieties to give a proof of Theorem~\ref{thm:tropicalequalstropicalized}.  Let the {\em incidence variety} be
\begin{equation}
\label{eqn:incidenceVariety}
W := \{(f_1, f_2, \dots, f_k, x) :  f_i(x) = 0 \text{ for all } i\} \subseteq \cprod_{i=1}^k (\CC^*)^{A_i} \times (\CC^*)^n,
\end{equation}
and let the {\em tropical incidence variety} be the set 
$$
\cT W := \{(F_1, F_2, \dots, F_k, X) : X \in \cT(F_i)\text{ for all } i\} \subseteq \cprod_{i=1}^k \RR^{A_i} \times \RR^n.
$$
The tropical incidence variety is the tropical prevariety \cite{BJSST} defined by the tropicalization of the polynomials $f_1, f_2, \dots, f_k$, where $f_i$ is considered as a polynomial in $m_i + n$ variables whose support in the $n$ variables is $A_i$ and whose $m_i$ terms have indeterminate coefficients.  Even if $A_i$ contains repeated points, the support of $f_i$ in $m_i + n$ variables has no repeated points.

\begin{lem}
\label{lem:incidence}
The polynomials $f_1, f_2, \dots, f_k$ form a tropical basis for the incidence variety $W$, i.e.\ the tropical incidence variety coincides with the support of the tropicalization of the incidence variety.
\end{lem}

\begin{proof} 
Let $K$ be the field of Puiseux series in $t$ with complex coefficients.  By the Fundamental Theorem of Tropical Geometry \cite{JMM, MaclaganSturmfels}, $\omega \in \cT(I) \cap \QQ^n$ if and only if $\omega = \val(x)$ for some $K$-valued point $x$ in the variety of $I$.  Since our fans are rational, it suffices to check that they agree on rational points.
Let $(F_1, F_2, \dots, F_k, X)$ be a rational point in the tropical prevariety, i.e.\ $F_1, F_2, \dots, F_k$ are (coefficient vectors of) tropical polynomials with support sets $A_1, A_2, \dots, A_k$, and $X \in \QQ^n$ is a tropical solution for each $F_i$.  We will show that this tuple can be lifted to a $K$-valued point in $W$, by first lifting $X$, then $F_1,F_2,\dots,F_n$. 
Let $x_0 = (t^{X_1}, t^{X_2}, \dots, t^{X_k}) \in (K^*)^n$.  
Then $F_i \in \QQ^{m_i}$ is contained in the tropical hypersurface of $f_{i}(x_0)$ considered as a polynomial in the indeterminate coefficients.  By the hypersurface case of the Fundamental Theorem (also known as Kapranov's Theorem) there is a tuple $c_i \in (K^*)^{m_i}$ of coefficients of $f_i$ with $\val(c_i) = F_i$ giving $f_{i}(x_0) = 0$.  Therefore $(F_1, F_2, \dots, F_k, X)$ can be lifted to the incidence variety and lies in the tropicalization of the incidence variety.
\end{proof}
A consequence of Lemma~\ref{lem:incidence} is that we may identify the tropical incidence variety with the support of the tropicalization of $W$ and we define its multiplicities accordingly.

The following lemma follows immediately from the definitions.  It is a tropical counterpart of an analogous statement for classical resultants.

\begin{lem}
\label{lem:projIncidence}
The tropical resultant is the projection of the tropical incidence variety onto the first factor.
\end{lem}

Let $\pi$ be the projection from $\RR^{m} \times \RR^{n}$, where the incidence variety lies, to the first factor $\RR^{m}$. 
We can now prove Theorem~\ref{thm:tropicalequalstropicalized}.

\begin{proof}[Proof of Theorem~\ref{thm:tropicalequalstropicalized}]
The resultant variety $\cR(\cA)$ is obtained from the incidence variety $W$ by projecting onto the first factor $\cprod_{i=1}^k (\CC^*)^{A_i}$ and taking the closure.  This proves the first of the following equalities.
$$\cT(\cR(\cA))=\cT(\overline{\pi(W)})=\pi(\cT(W))=\pi(\cT W)=\cT\cR(\cA)$$
The second follows from \cite{SturmfelsTevelev}  which says that the tropicalization of the closure of a projection of $W$ is the projection of the tropicalization of $W$. The third is Lemma~\ref{lem:incidence}, and the last is Lemma~\ref{lem:projIncidence}.
\end{proof}

For each $i = 1,2,\dots,k$, let $\widetilde{P_i}$ be the Newton polytope of $f_i$ in $\RR^{m_i} \times \RR^n$, which is in turn embedded in $\RR^{m} \times \RR^n$.   The tropical incidence variety is equal to the intersection $\cT(\widetilde{P_1})\cap \cdots \cap \cT(\widetilde{P_k})$, which is a union of normal cones of $\widetilde{P_1}+\cdots+\widetilde{P_k}$ associated to faces that are Minkowski sums of faces of dimension at least one.

The vertices of $\widetilde{P_1}, \dots, \widetilde{P_k}$ together linearly span an $m$-dimensional subspace in $\RR^m \times \RR^n$.  Projecting this onto $\RR^m$ takes each $\widetilde{P_i}$ isomorphically onto the standard simplex in $\RR^{m_i}$ which is embedded in $\RR^m$.  In particular, the Minkowski sum $\widetilde{P_1}+ \cdots+ \widetilde{P_k}$ projects isomorphically onto the Minkowski sum of standard simplices lying in orthogonal subspaces. It follows that every maximal cone in the tropical incidence variety appears uniquely as the intersection of some normal cones to edges of $\widetilde{P_1}, \widetilde{P_2}, \dots, \widetilde{P_k}$.  

The tropical incidence variety is
\begin{equation}
 \bigcup_{(E_1, E_2, \dots, E_k)}  \left( \bigcap_{i=1}^k \cN(\widetilde{E_i}) \right)
 \label{eqn:incidence}
\end{equation}
where the union runs over all choices of pairs $E_i$ of points from $A_i$ and $\cN(\widetilde{E_i})$ denotes the inner normal cone of the corresponding edge $\widetilde{E_i}$ in $\widetilde{P_i}$.  Even if the pair $E_i$ does not form an edge in the convex hull $Q_i$ of $A_i$, the pair $\widetilde{E_i}$ is always an edge of the simplex $\widetilde{P_i}$, so $\cN(\widetilde{E_i})$ has the right dimension.

\begin{lem}
Every maximal cone in the tropical incidence variety $\cT W=\cT(W)$ has multiplicity one.
\end{lem}

\begin{proof}
Since every vertex of every $\widetilde{P_i}$ has its own coordinate, the dimension of a face of the Minkowski sum $\widetilde{P_1}+\widetilde{P_2}+\cdots+\widetilde{P_k}$ minimizing a vector $\omega \in \RR^m \times \RR^n$ is the sum of the dimensions of the faces of each $\widetilde{P_i}$ with respect to $\omega$. The dimension of the incidence variety is $m+n-k$ and therefore, for a generic $\omega\in\cT(W)$, the face of $\widetilde{P_1}+\widetilde{P_2}+\cdots+\widetilde{P_k}$ minimizing $\omega$ has dimension $k$ and must be a zonotope. Consequently the forms $\init_\omega(f_1),\init_\omega(f_2),\dots,\init_\omega(f_k)$ are binomials, each with an associated edge vector $v_i\in\ZZ^{m+n}$. The vectors $v_1,v_2,\dots,v_k$ generate $C_\omega^\perp$ and after multiplying each $\init_\omega(f_i)$ by a monomial it ends up in $\langle\init_\omega(I)\rangle\subseteq\CC[\ZZ^{m+n}\cap C_\omega^\perp]$. Hence using the binomials to rewrite modulo $\langle\init_\omega(I)\rangle$ we get that $\dim_\CC(\CC[\ZZ^{m+n}\cap C_\omega^\perp]/\langle\init_\omega(I)\rangle)$ is bounded by the index of the sublattice generated by $v_1,v_2,\dots,v_n$ in $\ZZ^{m+n}\cap C_\omega^\perp$.
If we write the edge vectors as columns of a matrix, then the matrix contains a full-rank identity submatrix, so the sublattice has index one.
\end{proof}

 The tropical resultant is the projection of the tropical incidence variety, so
$$
\cT\cR(\cA) = \bigcup_{(E_1, E_2, \dots, E_k)} \pi \left(  \bigcap_{i=1}^k \cN(\widetilde{E_i}) \right).
$$

The {\em Cayley configuration} $\Cay(\cA)$ of a tuple  $\cA = (A_1, A_2, \dots, A_k)$ of point configurations in $\ZZ^{n}$ is defined to be the point configuration 
$$\Cay(\cA) = (\{e_1\} \times A_1) \cup \cdots \cup (\{e_k\} \times A_k)$$ in $\ZZ^{k} \times \ZZ^{n}$. We will also use $\Cay(\cA)$ to denote a matrix whose columns are points in the Cayley configuration.
See Example~\ref{ex:3tri}.

\begin{lem}
\label{lem:correspondence}
Let $E = (E_1, E_2, \dots, E_k)$ be a tuple of pairs from $A_1$, $A_2$, $\dots$, $A_k$ respectively.  Then the following cones coincide:
$$
 \pi \left( \bigcap_{i=1}^k \cN(\widetilde{E_i}) \right) = \RR_{\geq 0}\{e_{ij} : a_{ij} \notin E_i\} + \row(\Cay(\cA)).
$$
\end{lem}

\begin{proof}
Let $E$ be fixed.  The left hand side consists of tuples of tropical polynomials $(F_1, F_2, \dots, F_k) \in \cprod_{i=1}^k\RR^{A_i}$ for which there is a point $w \in \RR^n$ attaining the minimum for $F_i$ at $E_i$ for every $i$.

On the other hand, the cone $\RR_{\geq 0}\{e_{ij} : a_{ij} \notin E_i\}$ consists of all  $F = (F_1, F_2, \dots, F_k)$ such that the minimum for $F_i$ evaluated at the point $0 \in \RR^n$ has value $0$ and is attained at $E_i$ for every $i$. 
The tropical solution sets remains the same if coefficients of $F_i$ are changed by a tropical scalar multiple, which corresponds to adding to $F$ a multiple of the $i$-th row of $\Cay(\cA)$. 
For $w \in \RR^n$ and $F \in \RR^A$,
$$
F(x - w) = \min_{(c,a) \in F} c + a \cdot (x - w) = (F-w \cdot A)(x),\\
$$
where $A$ denotes the matrix whose columns are points in $A_1 \cup \cdots \cup A_k$, i.e.\ $A$ consists of the last $n$ rows of $\Cay(\cA)$, so
$$
\cT(F) + w = \cT(F - w A).
$$ 
Therefore, changing the coefficients $(F_1, F_2, \dots, F_k)$ by an element in the row space of $\Cay(\cA)$ has the effect of tropically scaling $F_i$'s and translating all the tropical solution sets together.  Thus the set on the right hand side consists of all tuples $(F_1, F_2, \dots, F_k)$ having a point $w \in \RR^n$ achieving the minimum for $F_i$ at $E_i$ for every $i$.
\end{proof}

The following result gives a simple description of the tropical resultant as a union of cones with multiplicities.  

\begin{thm}
\label{thm:orthants}
The tropical resultant of $\cA$ is the set 
\begin{equation}  
\label{eqn:orthants}
\cT\cR(\cA) =  \bigcup_{E} \RR_{\geq 0}\{e_{ij} : a_{ij} \notin E_i\} + \row(\Cay(\cA))
\end{equation}
where $E = (E_1, E_2, \dots, E_k)$ and each $E_i$ consists of two elements in $A_i$.  The multiplicity of the cone associated to $E$ is the index of the lattice spanned by the rows of $\Cay(E)$ in $\row(\Cay(E)) \cap \ZZ^m$.
\end{thm}
\noindent The set described in the right hand side of (\ref{eqn:orthants}) may not have a natural fan structure.
See Example~\ref{ex:3triLinks}(b).  

For a {\em generic} $\omega \in \cT\cR(\cA)$, we can compute the multiplicity of $\cT\cR(\cA)$ at $\omega$ as follows.  We say that $\omega$ is generic if all the cones on the right hand side of (\ref{eqn:orthants}) that contain $\omega$ are maximal-dimensional, contain $\omega$ in their relative interior, and have the same span.    Then the multiplicity at $\omega$ is the sum of multiplicities of the cones that contain $\omega$.  The generic points form a dense open set in $\cT\cR(\cA)$, and the lower-dimensional cones on the right hand side do not contribute to the multiplicity.

\begin{proof}[Proof of Theorem~\ref{thm:orthants}]
The set theoretic statement follows immediately from (\ref{eqn:incidence}) and Lemmas~\ref{lem:projIncidence} and~\ref{lem:correspondence}.

Let $\sigma = \bigcap_{i=1}^k \cN(\widetilde{E_i})$ be the cone corresponding to $E$ in the incidence variety, and $\tau = \pi(\sigma)$.
Using the refinement in \cite{CTY} of the multiplicity formula from tropical elimination theory \cite{SturmfelsTevelev}, the multiplicity of $\tau$ in the tropical resultant is the lattice index $[\LL_\tau : \pi(\LL_\sigma)]$, where $\LL_\tau = \RR \tau \cap \ZZ^m$ and $\LL_\sigma = \RR \sigma \cap \ZZ^{m+n}$.  The lattice $\LL_\sigma$ is defined by the following equations on $(c,x)\in\ZZ^{m+n}$
\begin{equation*}
\begin{split}
c \cdot (e_{ij} - e_{ik}) + x \cdot (a_{ij} - a_{ik}) & = 0
\mbox{ for } \{ a_{ij} ,  a_{ik} \} =   E_i
\end{split}
\end{equation*}
and is spanned by the integer points in the lineality space of the tropical incidence variety and the standard basis vectors $e_{ij}$ for $a_{ij} \notin E_i$.  The rows of the following matrix span the lattice points in the lineality space of the incidence variety:
$$
\left[ 
\begin{array}{c|c}
\Cay(\cA) &  
\begin{array}{c}
0 \\ -I_n
\end{array}
\end{array}
\right].
$$
Hence $\pi(\LL_\sigma)$ is spanned by the rows of $\Cay(\cA)$ and the $e_{ij}$'s for $a_{ij} \notin E_i$. 
\end{proof}

The first summand in (\ref{eqn:orthants}) plus the linear span of the first $k$ rows of $\Cay(\cA)$ is a tropical linear space obtained as a Cartesian product of tropical hyperplanes.  
Hence Theorem~\ref{thm:orthants} can be rephrased as follows.
Let $C$ be the matrix consisting of the first $k$ rows of $\Cay(\cA)$, so the kernel of $C$ is defined by equations of the form $c_{i,1} + c_{i,2} + \cdots + c_{i,m_i} = 0$ for $i = 1,2,\dots,k$.
Then
the tropical resultant is the set
\begin{equation}  
\label{eqn:prodLinSpaces}
\cT\cR(\cA) =  \cT(\ker(C)) + \row \left[ A_1 | A_2 | \cdots | A_k \right].
\end{equation}
The tropical linear space here is trivial to compute, as it is described by the first summand of (\ref{eqn:orthants}).  By contrast the tropical linear space computation required for tropical discriminants in \cite{DFS} can be challenging.   The state of the art in computing tropical linear spaces is the work of Rinc\'on \cite{Rincon}.

\begin{ex}
\label{ex:3tri}
Consider the tuple $\cA = (A_1, A_2, A_3)$ of the following point configurations in $\ZZ^2$:
\begin{equation}
\begin{split}
\blue{A_1} &= \{ (0,0), (0,1), (1,0) \}, \\
\green{A_2} &= \{ (0,0), (1,0), (2,1) \}, \\
\red{A_3} &= \{ (0,0), (0,1), (1,2) \}.
\end{split}
\end{equation}

\begin{center}
\includegraphics[scale=1.5]{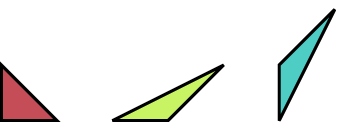}
\end{center}

The Cayley configuration $\Cay(\cA)$ consist of columns of the following matrix, which we also denote $\Cay(\cA)$:
$$ \Cay(\cA) = \left(
 \begin{array}{ccc|ccc|ccc}
      1&
      1&
      1&
      0&
      0&
      0&
      0&
      0&
      0\\
      0&
      0&
      0&
      1&
      1&
      1&
      0&
      0&
      0\\
      0&
      0&
      0&
      0&
      0&
      0&
      1&
      1&
      1\\
\hline
0&
      1&
      0&
      0&
      1&
      {2}&
      0&
      0&
      1\\
      0&
      0&
      1&
      0&
      0&
      1&
      0&
      1&
      {2}
      \end{array}
\right)
$$

The corresponding system of polynomials consist of
\begin{equation}
\begin{split}
\blue{f_1} &= c_{11} + c_{12} y + c_{13} x ,\\
\green{f_2} &= c_{21} + c_{22} x + c_{23} x^2 y ,\\
\red{f_3} &= c_{31} + c_{32} y + c_{33} x y^2 .
\end{split}
\end{equation}

The point 
$$
(0,0,0,0,1,5,0,1,5)
$$
is in the tropical resultant variety because the tropical hypersurfaces of the three tropical polynomials
\begin{equation}
\begin{split}
\blue{F_1} &= 0 \oplus X \oplus Y, \\
\green{F_2} &= 0 \oplus (1 \odot X) \oplus (5 \odot X^{\odot 2} \odot Y) \\
\red{F_3} &= 0 \oplus (1 \odot Y) \oplus (5 \odot X \odot Y^{\odot 2} )
\end{split}
\end{equation}
contain the common intersection points $(-1,-1)$ and $(-2,-2)$.  See Figure~\ref{fig:3trianglesTropCurvesA}.

\begin{figure}
\begin{center}
\includegraphics[scale=0.5]{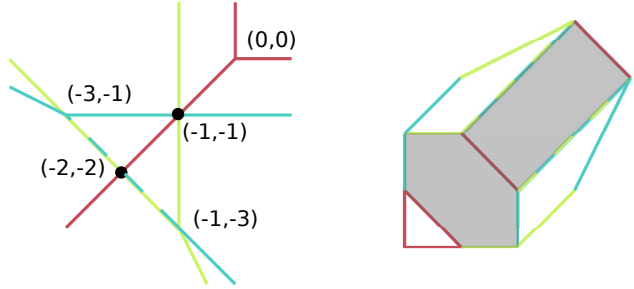}
\end{center}
\caption{A tropical hypersurface arrangement and its dual regular mixed subdivision (RMS) of the Minkowski sum of point configurations.  The mixed cells are shaded.  See Examples~\ref{ex:3tri} and~\ref{ex:3triLinks}(a).}
\label{fig:3trianglesTropCurvesA}
\end{figure}

Consider the incidence variety defined by the ideal
$$
I = \langle f_1, f_2, f_3 \rangle \subseteq \CC[c^{\pm 1}, x^{\pm 1}, y^{\pm 1}].
$$
The resultant variety is obtained by eliminating $x$ and $y$ from the system, i.e.\ it is defined by the ideal $I \cap \CC[c^{\pm 1}]$.  In this case, the resultant variety is a hypersurface defined by the resultant polynomial

\noindent
\begin{flushleft}
${c}_{1{2}}^{3} {c}_{{2}{3}}^{3} {c}_{{3}1}^{3}-2 {c}_{11}
      {c}_{1{2}}^{2} {c}_{{2}{3}}^{3} {c}_{{3}1}^{2}
      {c}_{{3}{2}}-{c}_{1{2}}^{2} {c}_{1{3}} {c}_{{2}{2}} {c}_{{2}{3}}^{2}
      {c}_{{3}1} {c}_{{3}{2}}^{2}+{c}_{11}^{2} {c}_{1{2}} {c}_{{2}{3}}^{3}
      {c}_{{3}1} {c}_{{3}{2}}^{2}-{c}_{1{2}} {c}_{1{3}}^{2} {c}_{{2}1}
      {c}_{{2}{3}}^{2} {c}_{{3}{2}}^{3}+{c}_{11} {c}_{1{2}} {c}_{1{3}}
      {c}_{{2}{2}} {c}_{{2}{3}}^{2} {c}_{{3}{2}}^{3}+3 {c}_{1{2}}^{2}
      {c}_{1{3}} {c}_{{2}1} {c}_{{2}{3}}^{2} {c}_{{3}1}^{2}
      {c}_{{3}{3}}+{c}_{11} {c}_{1{2}}^{2} {c}_{{2}{2}} {c}_{{2}{3}}^{2}
      {c}_{{3}1}^{2} {c}_{{3}{3}}+2 {c}_{1{2}}^{2} {c}_{1{3}}
      {c}_{{2}{2}}^{2} {c}_{{2}{3}} {c}_{{3}1} {c}_{{3}{2}}
      {c}_{{3}{3}}-{c}_{11} {c}_{1{2}} {c}_{1{3}} {c}_{{2}1}
      {c}_{{2}{3}}^{2} {c}_{{3}1} {c}_{{3}{2}} {c}_{{3}{3}}-{c}_{11}^{2}
      {c}_{1{2}} {c}_{{2}{2}} {c}_{{2}{3}}^{2} {c}_{{3}1} {c}_{{3}{2}}
      {c}_{{3}{3}}+2 {c}_{1{2}} {c}_{1{3}}^{2} {c}_{{2}1} {c}_{{2}{2}}
      {c}_{{2}{3}} {c}_{{3}{2}}^{2} {c}_{{3}{3}} -2 {c}_{11} {c}_{1{2}}
      {c}_{1{3}} {c}_{{2}{2}}^{2} {c}_{{2}{3}} {c}_{{3}{2}}^{2}
      {c}_{{3}{3}}-{c}_{1{2}}^{2} {c}_{1{3}} {c}_{{2}{2}}^{3} {c}_{{3}1}
      {c}_{{3}{3}}^{2}+3 {c}_{1{2}} {c}_{1{3}}^{2} {c}_{{2}1}^{2}
      {c}_{{2}{3}} {c}_{{3}1} {c}_{{3}{3}}^{2}-{c}_{11} {c}_{1{2}}
      {c}_{1{3}} {c}_{{2}1} {c}_{{2}{2}} {c}_{{2}{3}} {c}_{{3}1}
      {c}_{{3}{3}}^{2}-{c}_{11}^{3} {c}_{{2}1} {c}_{{2}{3}}^{2} {c}_{{3}1}
      {c}_{{3}{3}}^{2}-{c}_{1{2}} {c}_{1{3}}^{2} {c}_{{2}1}
      {c}_{{2}{2}}^{2} {c}_{{3}{2}} {c}_{{3}{3}}^{2}+{c}_{11} {c}_{1{2}}
      {c}_{1{3}} {c}_{{2}{2}}^{3} {c}_{{3}{2}} {c}_{{3}{3}}^{2}+{c}_{11}
      {c}_{1{3}}^{2} {c}_{{2}1}^{2} {c}_{{2}{3}} {c}_{{3}{2}}
      {c}_{{3}{3}}^{2}-{c}_{11}^{2} {c}_{1{3}} {c}_{{2}1} {c}_{{2}{2}}
      {c}_{{2}{3}} {c}_{{3}{2}} {c}_{{3}{3}}^{2}+{c}_{1{3}}^{3}
      {c}_{{2}1}^{3} {c}_{{3}{3}}^{3} -2 {c}_{11} {c}_{1{3}}^{2}
      {c}_{{2}1}^{2} {c}_{{2}{2}} {c}_{{3}{3}}^{3}+{c}_{11}^{2} {c}_{1{3}}
      {c}_{{2}1} {c}_{{2}{2}}^{2} {c}_{{3}{3}}^{3}$.
\end{flushleft}
      
\noindent It is homogeneous with respect to the rows of $\Cay(\cA)$.   Its Newton polytope is four-dimensional, has f-vector $(15,40,38,13,1)$ and lies in an affine space parallel to the kernel of $\Cay(\cA)$.

The tropical resultant is an eight-dimensional fan in $\RR^9$ with a five-dimensional lineality space $\row(\Cay(\cA))$.  As a subfan of the secondary fan of $\Cay(\cA)$, it consists of 89 (out of 338) eight-dimensional secondary cones, which can be coarsened to get the 40 normal cones dual to edges of the resultant polytope.  In other words, the 40 normal cones can be subdivided to obtain the 89 secondary cones.

In this example, the point configuration is essential, so $\cT(\cR\cA)$ is equal to the tropical discriminant of $\Cay(\cA)$, which is described in \cite{DFS} as
$$
\cT(\ker \Cay(\cA)) + \row(\Cay(\cA)).
$$
With the Gr\"obner fan structure, the tropical linear space $\cT(\ker \Cay(\cA))$ is a $4$-dimensional fan with f-vector $(1,15,66,84)$, so in the reconstruction of the Newton polytope, we have to process $84$ maximal cones, compared with $27$ cones from our description in (\ref{eqn:orthants}) or (\ref{eqn:prodLinSpaces}).  For larger examples, computing tropical linear spaces becomes a challenging problem, while our description remains simple.  In both cases, however, the main computational difficulty is the reconstruction of the Newton polytope from the tropical hypersurface.
\end{ex}

\subsection{Secondary fan structure and links in tropical resultants}
\label{sec:secondary}
  
 Let $A \in \ZZ^{d \times m}$ be an integer matrix with columns $a_1, a_2, \dots, a_m \in \ZZ^d$.  We will also denote by $A$ the point configuration $\{a_1, a_2, \dots, a_m\}$.  We allow repeated points in $A$, as we consider the points to be labeled by the set $\{1,2,\dots,m\}$, and every column of $A$ gets a distinct label.
 
Following \cite[Section~7.2A]{GKZ}, a {\em subdivision} of  $A$ is defined as a family $\Delta = \{C_i \subseteq A :  i \in I\}$ of subsets of $A$ such that
\begin{enumerate}
 \item $\dim(\conv(C_i)) = \dim(\conv(A))$ for each $i \in I$, 
 \item $\conv(A) = \bigcup_{i\in I} \conv(C_i)$, and
 \item for every $i,j \in I$, the intersection of $\conv(C_i)$ and $\conv(C_j)$ is a face of both, and $C_i \cap \conv(C_j) = C_j \cap \conv(C_i)$.
\end{enumerate}
 This notion is also called a {\em marked subdivision} by some authors, as it depends not only on the polyhedra $\conv(C_i)$ but also on the labeled sets $C_i$.  The elements in $\bigcup_{i\in I} C_i$ are called {\em marked}.
If $F$ is a face of $\conv(C_i)$ for some $C_i \in \Delta$, then the labeled set $C_i \cap F$ is called a {\em cell} of the subdivision.  The sets $C_i$'s are maximal cells. 

For two subdivisions $\Delta$ and $\Delta'$ of $A$, we say that $\Delta$ {\em refines} $\Delta'$ or $\Delta'$ {\em coarsens} $\Delta$ if every $C_i \in \Delta$ is contained in some $C_j' \in \Delta'$.  A subdivision is a {\em triangulation}  if no proper refinement exists, and equivalently, if every maximal cell contains exactly $\dim(\conv(A))+1$ elements.
 
Let $\omega : A \rightarrow \RR$ be an arbitrary real valued function on $A$, called a {\em weight vector}.  We can define a subdivision of $A$ {\em induced by} $\omega$ as follows.  Consider the unbounded polyhedron $P = \conv \{(a,\omega(a))\} + \RR_{\geq 0} \{ e_{d+1} \}$ in $\RR^{d+1}$, and let $\{ F_i : i\in I\}$ be its bounded facets.  Then the induced subdivision is $\{C_i : i\in I\}$ where $C_i = \{a \in A : (a, \omega(a)) \in F_i \}$. 
A subdivision $A$ is {\em regular} or {\em coherent} if it is induced by some weight vector $\omega$.
The partition of the space of weight vectors $\RR^A$ according to induced subdivisions is a fan, called the {\em secondary fan} of $A$.
  
 Following \cite[Section~7.1D]{GKZ}, we can construct the secondary polytope of $A$ as follows.  For a triangulation $T$ of a point configuration $A$, define the {\em GKZ-vector} $\phi_T \in \RR^A$ as
 $$\phi_T(a) := \sum_{\sigma\in T:a \in \sigma}\vol(\sigma)$$
 where the summation is over all maximal cells $\sigma$ of $T$ containing $a$.
 
 \begin{defn}
 The {\em secondary polytope} $\Sigma(A)$ is the convex hull in $\RR^A$ of the vectors $\phi_T$ where $T$ runs over all triangulations of $A$.
 \end{defn}
 
 \begin{thm}\cite[\textsection~7.1,  Theorem~1.7]{GKZ}
 The vertices of $\Sigma(A)$ are precisely the vectors $\phi_T$ for which $T$ is a regular triangulation of $A$. The normal fan of the secondary polytope $\Sigma(A)$ is the secondary fan of $A$.  The normal cone of $\Sigma(A)$ at $\phi_T$ is the closure of the set of all weights $w \in \RR^A$ which induce the triangulation $T$.
 \end{thm}

The {\em link} of a cone $\sigma \subseteq \RR^m$ at a point $v \in \sigma$ is
$$
\link_v(\sigma) = \{u \in \RR^m \,|\, \exists \delta > 0 :  \forall \varepsilon \text{ between } 0 \text{ and } \delta : v + \varepsilon u \in \sigma \}.
$$
The {\em link} of a fan $\cF$ at a point $v$ in the support of $\cF$ is the fan
$$
\link_v(\cF) = \{ \link_v(\sigma) \,|\, v \in \sigma \in \cF \}.
$$
For any cone $\tau \in \cF$, any two points in the relative interior of $\tau$ give the same link of the fan, denoted $\link_\tau(\cF)$. If a maximal cone $\tau\in\cF$ has an assigned multiplicity, then we let $\link_v(\tau)\in\link_v(\cF)$ inherit it.

We will first show that the link of the secondary fan at a point is a common refinement of secondary fans, or, more precisely, that a face of a secondary polytope is a Minkowski sum of secondary polytopes.
For a sub-configuration $C \subseteq A$, we can consider the secondary polytope of $C$ as embedded in $\RR^A$ by setting $\phi_T(a) = 0$ for $a \in A \backslash C$ for every triangulation $T$ of $C$.  On the other hand, the secondary fan of $C$ embeds in $\RR^A$ with lineality space containing the coordinate directions corresponding to $a \in A \backslash C$.

\begin{lem}
\label{lem:secondarylink}
Let $A$ be a point configuration, $\omega \in \RR^A$, and $\Delta_\omega$  be  the regular subdivision of $A$ induced by $\omega$.  
Then the face $F_\omega$ of the secondary polytope of $A$ supported by $\omega$ is the Minkowski sum of secondary polytopes of maximal cells in $\Delta_\omega$.
\end{lem}

\begin{proof}
Let $\omega' \in \RR^A$ be a generic weight vector and $p$ be the vertex of the Minkowski sum picked out by $\omega'$.  For all sufficiently small $\varepsilon > 0$, the triangulation $\Delta_{\omega + \varepsilon \omega'}$ refines the subdivision $\Delta_{\omega}$.  Let $p_i$ be the GKZ-vector of the triangulation of the $i$-th maximal cell induced by (the restriction of) the vector $\omega+\varepsilon\omega'$, which is the same as the triangulation induced by $\omega'$ because $\omega$ induces the trivial subdivision on each cell of $\Delta_\omega$.  Then the GKZ-vector of $\Delta_{\omega + \varepsilon\omega'}$ is the sum $\sum_i p_i$.  Hence the vertex of $F_\omega$ in direction $\omega'$ is $\sum_i p_i$. We can then conclude that the two polytopes are the same since they have the same vertex in each generic direction.
\end{proof}

We now define \emph{mixed subdivisions} as in \cite{TriangulationsBook}.
For point configurations $A_1, A_2, \dots, A_k$ in $\RR^n$, with $A_i = \{a_{i,j} : 1 \leq j \leq m_i \}$, the {\em Minkowski sum}
$$
\sum_{i=1}^k A_i = \{ a_{1,j_1} + a_{2,j_2} + \cdots + a_{k, j_k} :  1 \leq j_i  \leq m_i\ \} 
$$
 is a configuration of $m_1 m_2 \cdots m_k$ points labeled by $[m_1] \times [m_2] \times \cdots \times [m_k]$.  

\begin{defn}
A subset of labels is a {\em mixed cell} if it is a product of labels $J_1 \times J_2 \times \cdots \times J_k$ where $J_i$ is a nonempty subset of $[m_i]$, and it is {\em fully mixed} if in addition $J_i$ contains at least two elements for every $i = 1,2,\dots,k$.  A subdivision of the Minkowski sum $\sum_{i=1}^k A_i$ is {\em mixed} if every maximal cell is labeled by a mixed cell.
\end{defn}

A mixed subdivision of $\sum_{i=1}^k A_i$ is also referred to as a mixed subdivision of the tuple $\cA = (A_1, A_2, \dots, A_k)$.
Our definition of {\em fully mixed cell} differs from that of \cite[Section~6]{DFS}  where it is required that  $\conv(a_{i,j} : j \in J_i)$ has affine dimension at least one, while we only require that $J_i$ contains at least two elements.  These two definitions coincide if none of the $J_i$'s contains repeated points.

A mixed subdivision is called {\em regular} if it is induced by a weight vector 
$$w : \sum_{i=1}^k A_i \rightarrow \RR, \mbox{ where }  w : \sum_{i=1}^k a_{i,j_i} \mapsto \sum_{i=1}^k w_{i,j_i}$$
for some $(w_1, w_2, \dots, w_k) \in \RR^{m_1}\times\RR^{m_2}\times \cdots \times \RR^{m_k}$.
In \cite{Sturmfels94} a regular mixed subdivision (RMS) is also called a \emph{coherent mixed decomposition}.

\begin{thm} \cite[Theorem~5.1]{Sturmfels94}
For a subdivision  $\Delta$ of $\Cay(\cA)$, the collection of mixed cells of the form 
$\sum_{i=1}^k C_i$ such that  $C_i \subseteq A_i$ and $\bigcup_{i=1}^k C_i$ is a maximal cell of $\Delta$ forms a mixed subdivision of $\sum_{i=1}^k A_i$.  
This gives a one-to-one correspondence between the regular subdivisions of $\Cay(\cA)$ and RMSs of $\sum_{i=1}^k A_i$.  Moreover the partition of weight vectors $(w_1, w_2, \dots, w_k) \in \RR^{m_1}\times\RR^{m_2}\times \cdots \times \RR^{m_k}$ according to the induced RMS coincides with the secondary fan of $\Cay(\cA)$.
\end{thm}

From our description of tropical resultants, we get the following result which was proven for the resultant hypersurfaces in \cite[Theorem~5.2]{Sturmfels94} and stated for the {\em essential} configurations with no repeated points in \cite[Proposition~6.8]{DFS}.  See Remark~\ref{rem:matroid} for a definition of {\em essential}.
\begin{thm}
\label{thm:dual}
The tropical resultant is a subfan of the secondary fan of the Cayley configuration $\Cay(A_1,A_2,\dots,A_k)$, consisting of the cones whose corresponding mixed subdivision contains a fully mixed cell. 
\end{thm}
The multiplicities of secondary cones in the tropical resultant will be computed in Proposition~\ref{prop:secondaryMult} below.

\begin{proof}
For a tropical polynomial $F \in \RR^A$ the tropical solution set $\cT(F)$ is dual to the cells with at least two elements in the subdivision of $A$ induced by the coefficients of $F$.  More precisely, by the definition of tropical solution sets, $w \in \cT(F)$ if and only if $(1,w)$ is an inner-normal vector for the convex hull of lifted points $\{(c, a) \in \RR^{n+1} : c \odot x^a \text{ is a term in } F \}$ supporting at least two points of $A$.  The two points supported need not have distinct coordinates $a$.

Let $(F_1,F_2,\dots,F_k) \in \RR^{A_1}\times\RR^{A_2}\times \cdots \times \RR^{A_k}$.
The union of tropical solution sets $\bigcup_{i=1}^k \cT(F_i)$ inherits a polyhedral complex structure from the common refinement of the completions of $\cT(F_i)$ to $\RR^m$, which is dual to the RMS of $\cA$ induced by the coefficients of $(F_1,F_2,\dots,F_k)$. 
The tuple $(F_1,F_2,\dots,F_k)$ is in the tropical resultant if and only if the tropical solution sets have a common intersection, which holds if and only if there is a fully mixed cell in the dual RMS.
\end{proof}


 The tropical resultant is a subfan of the secondary fan. It is pure and connected in codimension one, so we can compute it by traversing, as in \cite{BJSST}.  To traverse the resultant fan, we need to know how to find the link of the fan at a cone. 

\begin{pro}
\label{prop:linkdecomposition}
Let $\cA = (A_1, A_2, \dots, A_k)$.  The support of the link at a point $\omega$ of the tropical resultant $\cT\cR(\cA)$ is a union of tropical resultants corresponding to sub-configurations of fully mixed cells in the RMS $\Delta_\omega$ of $\cA$ induced by $\omega$.
\end{pro}

\begin{proof}
By definition, a point $u$ is in the link if and only if $\omega + \varepsilon u$ induces a RMS with a fully mixed cell for all sufficient small $\varepsilon > 0$.  This happens if and only if at least one of the fully mixed cells in $\Delta_\omega$ is subdivided by $u$ into a RMS with a fully mixed cell, i.e.\ $u$ is in the tropical resultant of the sub-configurations of fully mixed cells.
\end{proof}

\begin{ex}
\label{ex:3triLinks}
Let $\cA$ be as in Example~\ref{ex:3tri}.  
\begin{itemize}
\item[(a)] The link at the point $(0,0,0,0,1,5,0,1,5)$ of the tropical resultant is a union of two hyperplanes whose normal vectors are:
$$
(0,-1,1,-1,1,0,1,-1,0) \text{ and } (0,0,0,0,1,-1,0,-1,1)
$$
respectively.  They are the resultant varieties of the sub-configurations of the two fully  mixed cells in Figure~\ref{fig:3trianglesTropCurvesA}.
\item[(b)] The link at the point $(0,0,0,0,-1,-1,0,0,1)$ consists of four rays modulo lineality space.  Figure~\ref{fig:3trianglesTropCurvesB} shows the induced  mixed subdivision, which contains two fully mixed cells.  The resultant of one fully mixed cell consists of three rays (modulo lineality), and the resultant of the other fully mixed cell consists of two rays.  They overlap along a common ray. 
\end{itemize}
\begin{figure}
\includegraphics[scale=0.5]{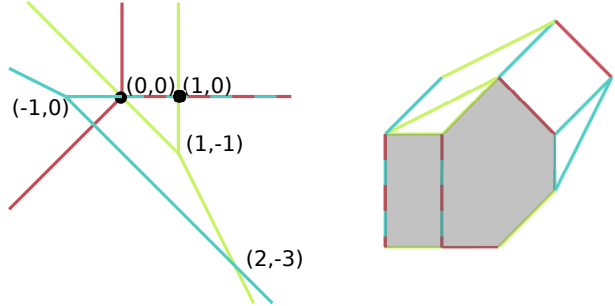}
\caption{The tropical solution sets at $(0,0,0,0,-1,-1,0,0,1)$ and the corresponding dual RMS in Example~\ref{ex:3triLinks}(b).}
\label{fig:3trianglesTropCurvesB}
\end{figure}
\end{ex}





The following lemma follows immediately from the definition of induced or regular subdivisions and shows that the description of the tropical resultant as a union of cones in Theorem~\ref{thm:orthants} is somewhat compatible with the secondary fan structure.  For any tuple $E = (E_1, E_2, \dots, E_k)$ of pairs $E_i \subset A_i$, let $C_E := \RR_{\geq 0}\{e_{ij} : a_{ij} \notin E_i\} + \row(\Cay(\cA))$ be the cone as in Theorem~\ref{thm:orthants}.

\begin{lem}
\label{lem:unionSecondary}
For each tuple $E$ as above, the cone $C_E$ is a union of secondary cones of $\Cay(\cA)$ corresponding to mixed subdivisions of $\sum_{i=1}^k A_i$ having a mixed cell containing $\sum_{i=1}^k E_i$.
\end{lem}

Let $\sigma$ be a secondary cone of $\Cay(\cA)$ which is a maximal cone in the tropical resultant $\cT\cR(\cA)$, and let $\Delta_\sigma$ be the corresponding regular mixed subdivision.  Then all the fully mixed cells in $\Delta_\sigma$ are of the form  $\sum_{i=1}^k E_i$ where each $E$ is a tuple of pairs as above. Otherwise $\sigma$ is not maximal in $\cT\cR(\cA)$.

\begin{pro}
\label{prop:secondaryMult}
The multiplicity of the tropical resultant $\cT\cR(\cA)$ at a secondary cone $\sigma$ of $\Cay(\cA)$ is the sum of multiplicities of cones $C_E$ (given in Theorem~\ref{thm:orthants})  over all tuples $E$ of pairs forming a mixed cells in $\Delta_\sigma$.
\end{pro}

\begin{proof}
By Lemma~\ref{lem:unionSecondary}, for each tuple $E$ of pairs, the cone $C_E$ contains $\sigma$ if and only if $\sum_{i=1}^k E_i$ is a mixed cell in $\Delta_\sigma$.  Otherwise $C_E$ is disjoint from the interior of $\sigma$. The multiplicity of $\sigma$ is the sum of multiplicities of $C_E$'s containing $\sigma$.
\end{proof}

The edges of the resultant polytope are normal to the maximal cones in the tropical resultant, and Proposition~\ref{prop:secondaryMult} can be used to find the lengths of the edges.  From this description, one can derive Sturmfels' formula \cite[Theorem~2.1]{Sturmfels94} for the vertices of the resultant polytope. 

\subsection{Tropical description of secondary polytopes}
\label{sec:tropSecondary}

We will give a tropical description of secondary polytopes of arbitrary point configurations and show how tropical resultants fit in.

\begin{pro}
Let $A$ be a $d\times m$ integer matrix whose columns affinely span an $r$-dimensional space.  The tropical hypersurface of the secondary polytope of the columns of $A$ is the set
$$
\mathop{\bigcup_{I \subset \{1,\dots,m\}}}_{|I| = r+2} \RR_{\geq 0} \{e_i : i \notin I\} + \row(A) + \RR\{\mathbbm{1}\}.
$$
where $\mathbbm{1}$ denotes the all-ones vector in $\RR^m$.
\end{pro}

\begin{proof}
Let $\omega \in \RR^m$, and let $\Delta_\omega$ be the regular subdivision of the columns of $A$ induced by $\omega$.  The weight vector $\omega$ is not in the tropical hypersurface of the secondary polytope if and only if $\Delta_\omega$ is not a triangulation, which happens if and only if there exists a maximal cell of $\Delta_\omega$ containing at least $r+2$ points of $A$.  For an $r+2$-subset $I$ of $\{1,\dots,m\}$, the cone  $\RR_{\geq 0} \{e_i : i \notin I\} + \row(A) + \RR\{\mathbbm{1}\}$ consists of all $\omega$ such that a cell of $\Delta_\omega$ contains $I$.  Note that $\row(A) + \RR\{\mathbbm{1}\}$ consists precisely of the weight vectors that induce the trivial subdivision of $A$ where there is a single maximal cell and all points are marked.
\end{proof}

Comparing with Theorem~\ref{thm:orthants}, we see that in the tropical description of the secondary polytope of $\Cay(\cA)$, the tropical resultant of $\cA$ is the union of the cones corresponding to the $I$'s with two points from each $A_i$.

\begin{ex}
\label{ex:secondary}
Let $A_1 = A_2 = \{0,1,2\}$ in $\ZZ$.  For $\cA = (A_1, A_2)$, the tropical hypersurface of the secondary polytope of $\Cay(\cA)$ and the tropical resultant of $\cA$ are depicted as graphs in Figure~\ref{fig:secondary}.  The resultant polytope has f-vector $(6, 11, 7, 1)$.  The secondary polytope in this case is combinatorially equivalent to the $3$-dimensional associahedron and has f-vector $(14, 21, 9, 1)$.   The first entry, the number of vertices of the polytope, is the number of connected components in the complement of the graph (the tropical hypersurface).  The third entry, the number of facets of the polytope,  can be seen as the number of crossings in this case.

\begin{figure}
\begin{center}
\includegraphics[scale=0.7]{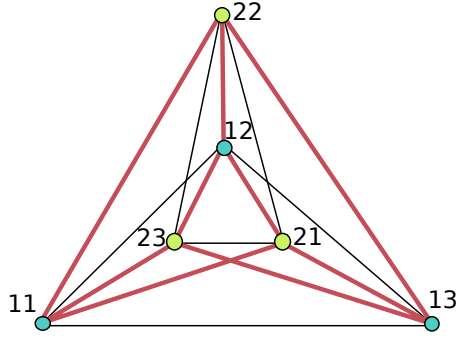}
\end{center}
\caption{A projective drawing of the tropical hypersurface of the secondary polytope of the Cayley configuration of two $1$-dimensional configurations in Example~\ref{ex:secondary}.  The tropical resultant is shown in color (or in bold).  A vertex labeled $ij$ represents the vector $e_{ij}$ in $\RR^6 = \RR^{A_1} \times \RR^{A_2}$, and an edge between $ij$ and $kl$ represents the cone $\RR_{\geq 0}\{e_{ij}, e_{kl}\} + \row(\Cay(\cA))$.  Compare with dual pictures in  \cite[Figure~2]{Sturmfels94} and \cite[Figure~3]{EFK}.}
\label{fig:secondary}
\end{figure}

\end{ex}

\subsection{Codimension of the resultant variety}
\label{sec:codimension}
In this section we discuss how to determine the codimension of the tropical resultant variety $\cT \cR(\cA)$. By the Bieri--Groves Theorem \cite{bierigroves} this is also the codimension of $\cR(\cA)$.

\begin{thm}
\label{thm:codimension}
The codimension of the tropical resultant equals $$k-\textup{Max}_E \textup{dim}(\sum_{i=1}^k \textup{conv}(E_i))$$ where each $E_i$ runs through all cardinality two subsets of $A_i$.
\end{thm}

\begin{proof}
The tropical resultant variety is the collection of all lifts of all
points in $\cA$ which give a fully mixed cell in the subdivision. Therefore
it is the closure of the collection of lifts which give a zonotope in
the mixed subdivision being a sum of convex hull of two points from each $A_i$. Let $P$
be such a zonotope and $E=(E_1,\dots,E_k)$ the $k$ 
pairs of points. We wish to find the dimension of the (relatively
open) cone $C_P$ of lifts which induces $P$. The height of the
first point of each $E_i$ may be chosen freely. The remaining $k$ points of $E$ must be lifted to the same subspace of dimension
$\textup{dim}(P)$, whose lift may be chosen with $\textup{dim}(P)$
degrees of freedom.  Finally, the height of the points not in $E$
maybe chosen generically as long as sufficiently large.  The codimension of $C_P$ is
therefore $k-\textup{dim}(P)$. The theorem follows since there are
only finitely many choices for $E$.
\end{proof}

\begin{lem}
\label{lem:span}
Let $L_i$ denote the subspace affinely spanned by $A_i$. The codimension of $\cR(\cA)$ only depends on the $L_i$'s and equals
$$k-\textup{Max}_{v\in\cprod_iL_i}\textup{dim}(\spann(v_1,\dots,v_k)).$$
\end{lem}

\begin{proof}
Since $\textup{conv}(E_i)\subseteq L_i$ the quantity of the lemma is smaller than or equal to that of Theorem~\ref{thm:codimension}. Conversely, if we have a collection $v\in\cprod_iL_i$ we now show how we can perform a sequence of changes to $v$ to make it only consist of vectors $v_i$ which are each differences between points of $A_i$ without lowering the dimension of $\spann(v_1,\dots,v_k)$. Consider a vector $v_i$. It is a linear combination of some $u_j$ where each $u_j$ is of the form $a_{is}-a_{it}$. If all $u_j$ belong to $W:=\spann(v_1,\dots,\widehat{v_i},\dots,v_k)$ then so will $v_i$ and it may be substituted by an arbitrary $u_j$ without lowering the dimension. If some $u_j$ does not belong to $W$ then substituting $u_j$ for $v_i$ will not lower the dimension.
\end{proof}

The proof also shows that instead of considering all line segments in Theorem~\ref{thm:codimension} it suffices to consider only a basis for the affine span for each $A_i$. This is useful while computing the codimension with this formula.

\begin{rem}
\label{rem:matroid}
We can define a matroid on a set of polytopes as follows.  A set of polytopes is {\em independent} if they contain independent line segments.  It is straightforward to check that the base exchange axiom holds.
The rank of the matroid is the maximal dimension of a fully mixed cell (a zonotope) spanned by two element subsets, one subset from each polytope.
The codimension of the tropical resultant equals the corank of the matroid, i.e.\ the number of polytopes minus the largest dimension of such a zonotope.  The (tropical) resultant variety is a hypersurface if and only if the matroid has corank one, which holds if and only if there is a unique circuit in the matroid. 
The tuple $\cA$ is {\em essential} \cite{Sturmfels94} if and only if this matroid of $k$ polytopes is uniform of rank $k-1$, that is, the unique circuit of the matroid consists of the entire ground set.
\end{rem}

Using Theorem~\ref{thm:codimension}, we get a new proof of Sturmfels' formula for the codimension of $\cR(\cA)$.  Recall that $Q_i$ is the convex hull of $A_i$.
\begin{thm}\cite[Theorem~1.1]{Sturmfels94}
\label{thm:sturmfelscodimension}
The codimension of the resultant variety $\cR(\cA)$ in $\cprod_{i=1}^k (\CC^*)^{m_i}$ is the maximum of the numbers $|I|-\textup{dim}(\sum_{i\in I}Q_i)$ where $I$ runs over all subsets of $\{1,\dots,k\}$.
\end{thm}

By the Bieri--Groves Theorem~\cite{bierigroves} and Theorem~\ref{thm:tropicalequalstropicalized}, the codimension of Theorem~\ref{thm:codimension} equals that of Theorem~\ref{thm:sturmfelscodimension}. In the following we explain how the equality of the two combinatorial quantities of Theorems~\ref{thm:codimension} and~\ref{thm:sturmfelscodimension} can also be seen as a consequence of Perfect's generalization (Theorem~\ref{thm:perfect}) of Hall's marriage theorem and Rado's theorem on independent transversals.

Let $S$ be the ground set of a matroid with rank function $\rho$.
Let $\cU = \{S_i : 1 \leq i \leq k\}$ be a family of subsets of $S$.  A subset $S'$ of $S$ is called an {\em independent partial transversal} of $\cU$ if $S'$ is independent and there exists an injection $\theta : S' \rightarrow \{1,2,\dots,k\}$ with $s \in S_{\theta(s)}$ for each $s \in S'$.

\begin{thm}(Perfect's Theorem \cite[Theorem~2]{Perfect})
\label{thm:perfect}
With the notation above, for every positive integer $d$, the family $\cU$ has an independent partial transversal of cardinality $d$ if and only if
$$
d \leq \rho(\cup_{i \in I} S_i) + k - |I|
$$
for every $I \subseteq \{1,2,\dots,k\}$.
\end{thm}
In particular, the maximum cardinality of an independent partial transversal is equal to the minimum of the numbers on the RHS of the inequality.

\begin{proof}[Proof of Theorem~\ref{thm:sturmfelscodimension}]
Let $S_i = \{a-b: a,b \in A_i\}$, $S = \bigcup_{i=1}^k S_i$, and $\cU = \{S_i : 1 \leq i \leq k\}$.  Consider the vector matroid on $S$ given by linear independence.  Then the quantity $\textup{Max}_E \textup{dim}(\sum_{i=1}^k \textup{conv}(E_i))$ is the cardinality of the maximal independent partial transversal of $\cU$.  By Perfect's Theorem, 
$$
\textup{Max}_E \textup{dim}(\sum_{i=1}^k \textup{conv}(E_i))
= \textup{Min}_{I \subseteq\{1,2,\dots,k\}} \dim(\sum_{i \in I} Q_i) + k - |I|.
$$
Hence the two quantities from Theorems~\ref{thm:codimension} and~\ref{thm:sturmfelscodimension} are equal.
\end{proof}



Straightforward evaluation of the formulas in Theorems~\ref{thm:codimension} and~\ref{thm:sturmfelscodimension} will require time complexity exponential in the input.  Moreover, the maximal bipartite matching problem is a special case of this codimension problem.

\begin{lem}
The maximal bipartite matching problem is reducible in polynomial time to the problem of computing codimension of resultants.
\end{lem}

\begin{proof}
Let $G$ be a bipartite graph with vertices $U \sqcup V$ and edges $E \subset U \times V$. Let $\{e_u : u \in U \}$ be the standard basis for $\RR^{U}$.  For each $v \in V$, let $A_v = \{e_u : (v,u) \in E\}$. 
Then the maximal cardinality of a bipartite matching in $G$ is equal to the dimension of the resultant variety of $\cA = (\{0\}\cup A_v : v \in V)$, and the size of $\cA$ is polynomial in the size of $G$.
\end{proof}


We use Theorem~\ref{thm:codimension} to construct an efficient algorithm for computing the codimension of a resultant.

\begin{thm}
\label{thm:polytime} The codimension of the resultant can be computed in polynomial time in the size of the input. 
\end{thm}

\begin{proof}
Let $\cA=(A_1, A_2, \dots, A_k)$ where each $A_i$ is a point configuration in $\ZZ^n$.
By Lemma~\ref{lem:span},  the codimension of $\cR(\cA)$ depends only on the linear spaces $L_1$, $L_2$, $\dots$, $L_k$ affinely spanned by $A_1, A_2, \dots, A_k$ respectively.  Choose a basis $B_i$ for each linear space $L_i$. Let $\cB =\{ B_1, B_2, \dots, B_k \}$ and $S = \bigcup_{i=1}^k B_i$.  A subset $S'$ of $S$ is called a {\em partial transversal} of $\cB$ if there is an injection $\theta : S' \rightarrow \{1,2,\dots,k\}$ with $s \in B_{\theta(s)}$.  The collection of partial transversals form an independent system of a matroid $\cM_1$ on ground set $S$, called the  {\em transversal matroid} of $\cB$.  Let $\cM_2$ be the vector matroid on $S$ defined by linear independence.  By Theorem~\ref{thm:codimension}, computing the codimension of the resultant is equivalent to computing the maximum cardinality of a linearly independent partial transversal, i.e.\ the largest subset of $S$ which is independent in both $\cM_1$ and $\cM_2$. 

We can use the cardinality matroid intersection algorithm \cite[Section~41.2]{SchrijverVolB} to find the maximum cardinality of a set independent in two matroids on the same ground set.  This algorithm is polynomial in the size of $S$ and the time for testing independence in the matroids.  Testing independence in $\cM_1$ can be reduced to the maximal bipartite matching problem and can be solved in polynomial time.  Testing linear independence in $\cM_2$ can be reduced to finding the rank of a matrix, which also takes polynomial time. 
\end{proof}

Alternatively, in the proof above, we can take $S$ to be the disjoint union of $B_1, B_2, \dots, B_k$, then the transversal matroid $\cM_1$ can be replaced by the partition matroid.  This was done in \cite{DyerGritzmannHufnagel} to prove that whether a mixed volume is zero can be decided in polynomial time.  This problem reduces to the codimension problem by observing that for a tuple $\cA = (A_1, A_2, \dots, A_n)$ of point configurations in $\ZZ^n$, the mixed volume of the convex hulls of $A_1, A_2, \dots, A_n$ is non-zero if and only if the codimension of $\cR(\cA)$ is zero, by Bernstein's Theorem.  

The algorithm described in Theorem~\ref{thm:polytime} is rather complex, but there is a simpler probabilistic or numerical algorithm.
For generic vectors $v_i \in L_i$ for $i = 1,2,\dots,k$, the codimension of the resultant is equal to $k - \rank( [v_1 | v_2 | \cdots | v_k] )$.
The challenge of turning this into a deterministic algorithm lies in making sure that the choices for $v_i$ are generic.  Our naive attempts at symbolic perturbations resulted in matrices whose ranks cannot be computed in polynomial time.

The polynomial time algorithm of Theorem~\ref{thm:polytime} can be used for finding a generic point in $\cT\cR(\cA)$ in polynomial time. Simply remove points from $\cA$ as long as the dimension does not drop. When we can no longer remove points, we have exactly two points left from each configuration of $\cA$. We then compute a generic point in $\cT\cR(\cA)$ using Theorem~\ref{thm:orthants}, possibly using a symbolic $\varepsilon$. It is unclear if a polynomial time algorithm exists for finding a generic point in specialized tropical resultants that we will see in Section~\ref{sec:special}.

\subsection{Traversing tropical resultants}
\label{sec:algorithms}
Tropical resultants are pure and connected in codimension 1. This allows the facets (maximal faces) to be enumerated via the well-known adjacency decomposition approach. By this we mean traversing the connected bipartite graph encoding the facet-ridge incidences of the fan.
Three operations are essential. We must be able to find some maximal cone in the fan, find the link at a ridge, and compute an adjacent maximal cone given a ray of the link at the ridge. In \cite{traversingsymmetricfans} these subcomputations were isolated in an oracle, and the author discussed a general algorithm for traversing a polyhedral fan (up to symmetry) represented only through oracle calls.  In the following paragraphs we will describe how to walk locally in the tropical resultant.  More details can be found in the next section for the more general setting of specialized tropical resultant.

To find a starting cone for the traversal, we use the description of the tropical resultant as a union of orthants plus a linear space, as described in Theorem~\ref{thm:orthants}.  Alternatively, a generic vector in a maximal cone of a resultant fan can be found in polynomial time using the algorithms for the codimension, as noted at the end of Section~\ref{sec:codimension}.

To find ridges, we compute facets of maximal cones.
To find the link of the tropical resultant at a point, we use the fact that the link at a point $\omega$ is a union of smaller tropical resultants associated to the fully mixed cells in the mixed subdivision of $\cA$ induced by $\omega$, as shown in Proposition~\ref{prop:linkdecomposition}.

In the tropical resultant, as a subfan of the secondary fan of $\Cay(\cA)$, each cone can be represented by a regular subdivision of $\Cay(\cA)$.
The smallest secondary cone containing a given vector $\omega$ can be constructed from the regular subdivision induced by $\omega$ as explained in \cite[Section~5.2]{TriangulationsBook}.

In our implementation we represent the regular subdivision $\Delta$ induced by $\omega$ by $\omega$ and the triangulation induced by a ``placing'' or ``lexicographic'' perturbation of $\omega$.  From this triangulation, we can easily recover the subdivision $\Delta$ by comparing the normal vectors of the maximal cells of the triangulation lifted by $\omega$.  For this to work, it is important to perturb $\omega$ in such a way that all the marked points in $\Delta$ remain marked in the refined triangulation.  A full triangulation of $\Cay(\cA)$ is computed from scratch only once at the beginning.  There are standard methods for computing a placing triangulation of any point configuration; see \cite[Section~8.2.1]{TriangulationsBook}.  To obtain a desired triangulation from a known triangulation, we find a path in the flip graph of regular triangulations and perform flips as in \cite[Section~8.3.1]{TriangulationsBook}.  This is the analogue of a Gr\"obner walk in the setting of secondary fans.  
 
 To find the secondary cone in the link at $u$ given by a ray $v$, we compute the subdivision induced by $u + \varepsilon v$ for sufficiently small $\varepsilon > 0$.  Such a vector $u + \varepsilon v$ is represented symbolically in a way similar to a matrix term order in the theory of Gr\"obner bases.

\section{Resultants with specialized coefficients}
\label{sec:special}

For some applications such as implicitization we need to compute resultant varieties while specializing some of the coefficients. 
This problem was studied in \cite{EKP, EFKP} for the case when the resultant variety is a hypersurface.  In that case, the Newton polytope of the specialized resultant is the projection of the resultant polytope, and the authors computed the projection of resultant polytopes using Sturmfels' formula for vertices of resultant polytopes \cite[Theorem~2.1]{Sturmfels94} and beneath-beyond or gift-wrapping methods for computing convex hulls.  
In our language, computing a projection of a polytope is equivalent to computing the restriction of the normal fan to a subspace. 

In tropical geometry, specialization of certain coefficients amounts to taking the {\em stable intersection} of the tropical resultant with certain coordinate hyperplanes. In this section we first define the specialized tropical resultants and then present algorithms for computing them.

A polyhedral complex in $\RR^n$ is called {\em locally balanced} if it is pure dimensional and the link at every codimension one face positively spans a linear subspace of $\RR^n$.

\begin{defn}
Let $\cF_1$ and $\cF_2$ be locally balanced fans in $\RR^n$. We define the stable intersection as the fan 
\begin{align*}\cF_1\stint\cF_2:= \{C_1\cap C_2: & (C_1,C_2)\in\cF_1\times\cF_2 \text{ and }\\
 &\supp(\link_{C_1}(\cF_1))-\supp(\link_{C_2}(\cF_2))=\RR^n\}
\end{align*}
with support
$$\supp(\cF_1\stint\cF_2)=\{\omega\in\RR^n:\supp(\link_\omega(\cF_1))-\supp(\link_\omega(\cF_2))=\RR^n\}.$$ 
If in addition $\cF_1$ and $\cF_2$ are balanced then the stable intersection inherits multiplicities from $\link_\omega(\cF_1)$ and $\link_\omega(\cF_2)$ as follows:
$$\mult_\omega(\cF_1\stint\cF_2):=$$
$$\sum_{C_1,C_2}\mult_{C_1}(\link_\omega \cF_1)\cdot\mult_{C_2}(\link_\omega\cF_2)\cdot[\ZZ^n:(\ZZ^n\cap\RR C_1)+(\ZZ^n\cap\RR C_2)]$$ where the sum runs over 
$C_1\in\link_\omega(\cF_1)$ and $C_2\in\link_\omega(\cF_2)$ such that $\omega'\in C_1-C_2$ for a fixed generic vector
 $\omega'\in\RR^n$.
\end{defn}
Notice that the support of $\cF_1\stint\cF_2$ depends only on $\supp(\cF_1)$ and $\supp(\cF_2)$. We will therefore extend the definition of stable intersections to intersections of supports of locally balanced fans and regard them as subsets of $\RR^n$.

For proofs of the following six statements, of which some are known to the community already, we refer to the upcoming paper \cite{stableIntersection}.

Orthogonally projecting a polytope onto a linear space is equivalent to stably intersecting the tropical hypersurface of the polytope with the linear space.  This is consistent with the fact that the Newton polytope of the specialized resultant is a projection of the resultant polytope onto a suitable coordinate subspace.

\begin{thm}
Let $P\subset \RR^n$ be a polytope, $L \subset \RR^n$ be a linear subspace, and $\pi : \RR^n \rightarrow L$ be the orthogonal projection.  Then 
$$\cT(\pi(P))=(\cT(P)\stint L)+ L^\perp.$$
\end{thm}

\begin{lem}
\label{lem:stableAssociative}
For any locally balanced fans $\cF_1$, $\cF_2$, and $\cF_3$, we have
\begin{enumerate}
\item $(\cF_1 \stint \cF_2) \stint \cF_3 = \cF_1 \stint (\cF_2 \stint \cF_3)$
\item $(\supp(\cF_1) \cup \supp(\cF_2)) \stint \supp(\cF_3)  = \supp(\cF_1 \stint \cF_3) \cup \supp(\cF_2 \stint \cF_3)$.
\end{enumerate}
\end{lem}

\begin{lem}
\label{lem:computingstableintersections}
For locally balanced fans $\cF_1$ and $\cF_2$
$$
\supp(\cF_1 \stint \cF_2) = \mathop{\bigcup_{C_1 \in \cF_1, C_2 \in \cF_2}}_{\codim(C_1 + C_2) = 0} C_1 \cap C_2 .
$$
\end{lem}

\begin{cor}
\label{cor:linkst}
For locally balanced fans $\cF_1$ and $\cF_2$
$$\link_\omega(\cF_1)\cap_{st} \link_\omega(\cF_2)=\link_\omega(\cF_1 \cap_{st} \cF_2).$$
\end{cor}

\begin{pro}
\label{prop:stabledimension}
The stable intersection of two locally balanced fans is either empty or a locally balanced fan whose codimension is the sum of the codimensions.
\end{pro}

\begin{lem}
\label{lem:slicingbinomial}
Let $I$ be an ideal in $\kk[x_1,x_2,\dots,x_n]$.  Then
$$
\supp(\cT(I)) \stint \{x: x_1 = 0\} = \supp(\cT(\langle I \rangle + \langle x_1 - \alpha \rangle))
$$
where $\langle I \rangle$ is the ideal in $\kk(\alpha)[x_1,x_2,\dots,x_n]$ generated by $I$.
\end{lem}

\begin{defn}
\label{def:specialized}
Let $S=(S_1,\dots,S_k)$ where each $S_i\subseteq \{1,\dots,m_i\}$ represent a choice of points in the configuration $\cA$. The coefficients of the monomials indexed by $S$ are called \emph{specialized}. Let $U_i:=\{x\in\RR^{m_i}:x_j=0 \text{ for all } \forall j\in S_i\}$ and $U_S :=\cprod_{i=1}^k U_i$. We define the \emph{specialized tropical resultant} as $$\cT\cR_S(A):=\cT\cR(\cA)\cap_{st}\{U_S\}.$$
\end{defn}

We will use the following proposition to justify the word ``specialized.'' Let $I$ be the ideal of $\cR(\cA)$ and $J$ be a new ideal generated by $I$ together with the binomials $c_j-\gamma_j$ 
 for specialized coefficients $c_j$, where $\gamma_j$'s are parameters. We define the specialized resultant variety $\cR_S(\cA):=V(J)\subseteq\cprod_{i=1}^k(K^*)^{m_i}$ as the variety defined by $J$ where $K$ is the field of rational functions in $\gamma_j$'s with coefficients in $\CC$.

When all the containments $S_i \subset \{1,\dots, m_i\}$ are strict, i.e.\ not all coefficients are specialized in any of the polynomials $f_i$, then the variety $\cR_S(\cA)$ is irreducible.  To see this, consider the {\em specialized incidence variety} $W_S$ cut out by the polynomials $f_1, \dots, f_k$ with some, but not all, coefficients specialized in each $f_i$.  For a fixed $x\in(\CC^*)^n$ each $f_i$ gives an affine constraint on the the non-specialized coefficients in $f_i$. Such constraints are solvable since each $x_j\not=0$ and they are simultaneously solvable since they concern different sets of coefficients. Hence $W_S$ is a vector bundle over $(\CC^*)^n$ and is irreducible.  Therefore its projection $\cR_S(\cA)$ is also irreducible.  


In general, for a prime ideal $I \subset \kk[x_1, \dots, x_n]$, with $\kk$ algebraically closed, and a generic $\alpha$, specializing a variable $x_1$ to $\alpha$ may not preserve primality, i.e.\ the ideal $I_1:=I + \langle x_1 - \alpha\rangle\subseteq\overline{\kk(\alpha)}[x_1,\dots,x_n]$ need not be prime.  However, all its irreducible components have the same tropical variety.  To see this, note that $\cT(I_1)=\{0\}^1\times \cT(I_2)$ where $I_2:=I\subseteq\overline{\kk(x_1)}[x_2,\dots,x_n]$.  The ideal $I_3:=I\subseteq\kk(x_1)[x_2,\dots,x_n]$ is prime because $I$ remains prime under extension from $\kk[x_1]$ to $\kk(x_1)$, as primality is preserved under localization.  Deciding whether a point is in a tropical variety can be done with reduced Gr\"obner bases which are independent of the field extension, so we have $\{0\}^1\times \cT(I_3)=\{0\}^1\times \cT(I_2)=\cT(I_1)$.  Furthermore,  since $I_1$ is prime, by \cite[Proposition~4]{CartwrightPayne}, all irreducible components of $I_2$ have the same tropicalization.  Since the tropical varieties of the irreducible components of $I_3$ are the same as those of the irreducible components of $I_2$, the conclusion follows.

\begin{pro}
\label{prop:special}
 The tropicalization of $\cR_S(\cA)$ is $\cT\cR_S(A)$.
\end{pro}

\begin{proof}
The statement follows from Lemmas~\ref{lem:slicingbinomial} and~\ref{lem:stableAssociative}(1).
\end{proof}

The computation of the tropicalization of $\cR_S(\cA)$ can be performed using Buchberger's algorithm as explained in \cite{BJSST} over the field of rational functions in the $\gamma_j$'s. During this computation finitely many polynomials in the $\gamma_j$'s appear as numerators and denominators of the coefficients. Substituting constant values for the $\gamma_j$'s will give the same computation unless one of these polynomials vanish. Hence specializing $\gamma_j$'s to values outside a hypersurface in $(\CC^*)^S$ will lead to a specialized tropical resultant variety. This explains the word ``specialized.''

If $\cT\cR_S(A)$ is nonempty, then its codimension can be computed using Proposition~\ref{prop:stabledimension} and the codimension formulas from Section~\ref{sec:codimension}. Thus it remains to give an algorithm for checking if the specialized resultant is empty.  Recall that $m := \sum_i m_i$ is the total number of points in $\cA$.

\begin{lem}
\label{lem:nonempty}
Let $\cA$ and $S$ be as in Definition~\ref{def:specialized}. Define the extended tuple $\cB=(B_1,\dots,B_k)$ where $B_i$ consists of $b_{i,j}= (a_{i,j} , v_{i,j} ) \in \ZZ^n \times \ZZ^{m-|S|}$, with $v_{i,j} \in\ZZ^{m-|S|}$ being $0$ if $j\in S_i$ and a standard basis vector otherwise. If the standard vector is chosen differently for every non-specialized coefficient then 
$$\cT\cR_S(A)\not=\emptyset \Leftrightarrow \cT\cR(\cB)=\RR^{m}.$$
\end{lem}

\begin{proof}
According to Lemma~\ref{lem:computingstableintersections}, $\cT\cR_S(A) \neq \emptyset$ if and only if there exists a cone $C\subseteq \cT\cR(A)$ such that $U_S+C=\RR^{m}$ where $U_S$ is as in Definition~\ref{def:specialized}.  According to the simple description of tropical resultants in Theorem~\ref{thm:orthants} we may assume that $C$ has the form $\RR_{\geq 0}\{e_{ij} : a_{ij} \notin E_i\} + \row(\Cay(\cA))$. Equivalently, the stable intersection is nonempty if and only if there exists a choice $E$ such that $\RR_{\geq 0}\{e_{ij} : a_{ij} \notin E_i\} + \row(\Cay(\cA))+U_S$ has dimension $m$. Applying Theorem~\ref{thm:orthants} to $\cB$, this is equivalent to $\cT\cR(\cB)$ being full dimensional, since $\row(\Cay(\cA))+U_S=\row(\Cay(\cB))$.
\end{proof}

Combining Lemma~\ref{lem:nonempty} and the results from Section~\ref{sec:codimension} about computation of codimension, we get a polynomial time algorithm for deciding if a specialized result is nonempty. Another consequence of the lemma is the following algorithm for checking membership of a point in a specialized tropical resultant.

\begin{alg}(SpecializedResultantContains($\cA,S,\omega$))\\
{\bf Input:} A tuple $\cA$ of point configurations and a choice $S$ of specialized coefficients. A vector $\omega\in\RR^{m}$.\\
{\bf Output:} ``True'' if $\omega\in\cT\cR_S(\cA)$, ``False'' otherwise.
\begin{itemize}
\item Compute the mixed subdivision of $A$ induced by $\omega$ by computing the regular subdivision of $\Cay(\cA)$ induced by $\omega$. (See Section~\ref{sec:algorithms}).
\item For each fully mixed cell:
\begin{itemize}
\item construct a subconfiguration $\cA'$ of points involved in the cell.
\item Return ``True'' if the specialized resultant of $\cA'$ is nonempty.
\end{itemize}
\item Return ``False''.
\end{itemize}
\end{alg}
\begin{proof}
By Lemma~\ref{lem:stableAssociative}(2), Proposition~\ref{prop:linkdecomposition} and Corollary~\ref{cor:linkst}, we have that the support of $\link_\omega(\cT\cR_S(\cA))$ is the union of supports of  $\cT\cR_S(\cA')$, under the appropriate identification of $\cT\cR_S(\cA')$ as a subset of $\RR^m$, where $\cA'$ runs over all fully mixed cells of the mixed subdivision of $\cA$ induced by $\omega$.  Hence $\omega \in \cT\cR_S(\cA)$ if and only if one of  $\cT\cR_S(\cA')$ is nonempty.
\end{proof}


\begin{alg}(NonTrivialVectorInSpecializedResultant($\cA,S$))\label{alg:vectorinstable}\\
{\bf Input:} A tuple $\cA$ of configurations, a choice $S$ of specialized coefficients such that $U_S\cap\row(\Cay(\cA))\subsetneq\cT\cR_S(\cA)$.\\
{\bf Output:} A vector $\omega \in \cT\cR_S(\cA)\setminus\row(\Cay(\cA))$
\begin{itemize}
\item For each $E = (E_1, E_2, \dots, E_k)$ : $E_i$ is a two-element subset of $A_i$,
  \begin{itemize}
  \item Let $C = \RR_{\geq 0}\{e_{i,j} : i \notin E_j\} + \row(\Cay(\cA))$.
  \item If $\codim(C + U_S) = 0$ and $U_S\cap C\not=U_S\cap\row(\Cay(\cA))$ then
    \begin{itemize}
      \item Find among the generators of $U_S\cap C$ a vector $v$ outside the subspace $U_S\cap\row(\Cay(\cA))$.
    \item Return $v$.
  \end{itemize}
  \end{itemize}
\end{itemize}
\end{alg}

The following recursive algorithm finds a perturbed point in a starting cone for the specialized tropical resultant $\cT\cR_S(\cA)$.

\begin{alg}
\label{alg:startingpoint}
(StartingPoint($\cA,S$))\\
{\bf Input:} A tuple $\cA$ of configurations, a choice $S$ of specialized coefficients such that $\cT\cR_S(\cA)\not=\emptyset$.\\
{\bf Output:} A vector $\omega_\varepsilon\in\QQ(\varepsilon)^{m}$ such that for every fan structure of $\cT\cR_S(\cA)$ defined over $\QQ$ it holds that for $\varepsilon>0$ sufficiently small, $\omega_\varepsilon$ is in a maximal cone of $\cT\cR_S(\cA)$.
\begin{itemize}
\item If $\dim(\cT\cR_{S}(\cA)) = \dim(U_S\cap\row(\Cay(\cA)))$, then return $b_1+\varepsilon b_2+\cdots+\varepsilon^{t-1}b_t$ where $b_1,b_2,\dots,b_t$ is some basis of $U_S\cap\row(\Cay(\cA))$.
\item Compute an $\omega\in\cT\cR_S(\cA)\setminus\row(\Cay(\cA))$ using Algorithm~\ref{alg:vectorinstable}.
\item Compute the subdivision $\Delta_\omega$ of $\Cay(\cA)$ induced by $\omega$.
\item For every fully mixed cell in $\Delta_\omega$.
\begin{itemize}
\item Let $\cA'$ be the subconfiguration of the involved points.
\item Let $S'$ be the restriction of $S$ to $\cA'$.
\item If $\textup{codimension}(\cT\cR_{S'}(\cA'))=\textup{codimension}(\cT\cR_S(\cA))$ then
\begin{itemize}
\item Return $\omega+\varepsilon\cdot\textup{StartingPoint}(\cA',S')$.
\end{itemize}
\end{itemize}
\end{itemize}
\end{alg}
\begin{proof}
The correctness of the algorithm follows from the facts that the link at $\omega$ of the tropical resultant is the union of tropical resultants corresponding to the fully mixed cells in $\Delta_\omega$  (Proposition~\ref{prop:linkdecomposition}), that taking links commutes with taking stable intersections (Corollary~\ref{cor:linkst}), and that the returned value from the recursive call (after expansion with zeros) is a generic vector in the link of $\cT\cR_S(\cA)$ at $\omega$ and in particular lies outside the secondary cone of $\omega$.
\end{proof}


We now turn to the problem of enumerating all maximal cones in $\cT\cR_S(\cA)$ considered as a subfan of the restriction of the secondary fan of $\Cay(\cA)$ to the subspace $U_S$.  While connectedness in codimension~$1$ is not preserved under stable intersections in general, a specialized tropical resultant $\cT\cR_S(\cA)$ is connected in codimension~$1$ because it coincides with the tropical variety of a prime ideal, as shown in the paragraph above Proposition~\ref{prop:special}.
The proof in \cite{BJSST} that the tropical varieties of prime ideals are connected in codimension~$1$ contained some mistakes, which were later corrected in~\cite{CartwrightPayne}.

The output of Algorithm~\ref{alg:startingpoint} can be converted into a secondary cone in $\cT\cR(\cA)$ containing $\omega_\varepsilon$ in its relative interior, for example by computing a maximal secondary cone containing $\omega_\varepsilon$ and taking the face containing $\omega_\varepsilon$ in its relative interior.  For sufficiently small $\varepsilon > 0$, this secondary cone would not change with $\varepsilon$.

Following the approach of \cite{traversingsymmetricfans} discussed in Section~\ref{sec:algorithms}, we are left with the problem of computing the link at a ridge in $\cT\cR_S(\cA)$. If the subspace $U_S$ is generic enough such that $$\codim(U_S \cap \row(\Cay(\cA)) = \codim(U_S) + \codim(\row(\Cay(\cA))),$$ then the link of $\cT\cR_S(\cA)$ is combinatorially equivalent to the link of $\cT\cR(\cA)$ and the support of the link is a union of resultant fans of subconfigurations (Proposition~\ref{prop:linkdecomposition}) where each fan can be found using Theorem~\ref{thm:orthants}. If $U_S$ is not generic, then computing a stable intersection with $U_S$ is required for finding the link in $\cT\cR_S(\cA)$ (Corollary~\ref{cor:linkst}). This is Algorithm~\ref{alg:stablelinkprojections} below.  Recall that the dimension of $\cT\cR_S(\cA)=\cT\cR(\cA)\cap_{st}\{U_S\}$ can be computed using Proposition~\ref{prop:stabledimension} and the codimension formulas from  Section~\ref{sec:codimension}.

\begin{alg} \label{alg:stablelinkprojections}
StableLink($\cA,S,\omega$)\\
{\bf Input:} A tuple $\cA$ of configurations, a choice $S$ of specialized coefficients, a vector $\omega\in\RR^n$ in the relative interior of a ridge $R$ of $\cT\cR_S(\cA)$.\\
{\bf Output:} A vector in each facet of $\link_\omega(\cT\cR_S(\cA))$.
\begin{itemize}
\item Let $d$ be the dimension of $\cT\cR(\cA)\cap_{st}\{U_S\}$.
\item Compute the subdivision $\Delta_\omega$ of $\Cay(\cA)$ induced by $\omega$.
\item $l:=\emptyset$.
\item For every fully mixed cell in $\Delta_\omega$
\begin{itemize}
\item Let $\cA'$ be the subconfiguration of involved points in the cell.
\item For each $E=(E_1,E_2,\dots,E_k):E_i$ is a two-element subset of $\cA'_i,$
\begin{itemize}
\item Let $C=\RR_{\geq 0}\{e_{i,j} : i \notin E_j\} + \row(\Cay(\cA))$.
\item If $\dim(U_S+C)=m$ and $\dim(U_S\cap C)=d$ then
\begin{itemize}
\item   Let $V$ be a set of one or two vectors in $U_S \cap C$ such that $(U_S \cap C)+\spann(R)$ is positively spanned by $V \cup \spann(R)$.  
\item  $l:=l\cup V$
\end{itemize}
\end{itemize}
\end{itemize}
\item Return $l$.
\end{itemize}
\end{alg}

Another approach to computing a link at a point of the stable intersection is to compute the restriction of the secondary fan of each fully mixed subconfiguration to $U_S$. We then get the resultant fan as certain rays of the secondary fan. This is Algorithm~\ref{alg:stablelinkslicing}.

\begin{alg} \label{alg:stablelinkslicing}
StableLink($\cA,S,\omega$)\\
{\bf Input:} A tuple $\cA$ of configurations, a choice $S$ of specialized coefficients, a vector $\omega\in\RR^n$ in the relative interior of a ridge $R$ of $\cT\cR_S(\cA)$.\\
{\bf Output:} A vector in each facet of $\link_\omega(\cT\cR_S(\cA))$.
\begin{itemize}
\item Let $d$ be the dimension of $\cT\cR(\cA)\cap_{st}\{U_S\}$.
\item Compute the subdivision $\Delta_\omega$ of $\Cay(\cA)$ induced by $\omega$.
\item $l:=\emptyset$.
\item For every fully mixed cell in $\Delta_\omega$
\begin{itemize}
\item Let $\cA'$ be the subconfiguration of the involved points of the cell.
\item If the codimension of the lineality space of the restriction $\cF$ of the secondary fan of $\Cay(\cA')$ to $U_S$ is $m-d$, then
\begin{itemize}
\item Choose $v$ such that $v$ extends $\spann(R)\cap U_S$ to a generating set of the lineality space of $\cF$.
\item If SpecializedResultantContains($\cA',S,v$) then $l:=l\cup\{v,-v\}$.
\end{itemize}
\item else
\begin{itemize} 
\item Compute all maximal cones in $\cF$ (by traversal).
\item For each ray $v$ in $\cF$, if SpecializedResultantContains($\cA',S,v$) then $l:=l\cup\{v\}$.
\end{itemize}
\end{itemize}
\item Return $l$.
\end{itemize}
\end{alg}
The above algorithm is to be read with proper identifications.  Namely, when restricting to $\cA'$ the vectors in $\RR^m$ need to be truncated accordingly, and so does the set $S$, and $v$ needs to be expanded when adding it to $l$.
When adding vectors to $l$, it is advantageous to choose the vectors as primitive vectors orthogonal to the span of the ridge so that duplicates can be removed easily.

If $U_S$ is of high dimension, a typical situation is that each subconfiguration is a number of edges and a triangle. In this case there are only few choices $E$ to run through in Algorithm~\ref{alg:stablelinkprojections}. For lower dimensional $U_S$ there can be many choices of $E$ but with many of the contributions to the stable intersection being the same. See Example~\ref{ex:badlink}. In such a case Algorithm~\ref{alg:stablelinkslicing} performs better than Algorithm~\ref{alg:stablelinkprojections}. In general it is difficult to predict which algorithms is better. In our implementation we use mostly Algorithm~\ref{alg:stablelinkslicing}, and Algorithm~\ref{alg:stablelinkprojections} only when there is no specialization.
\begin{ex}
\label{ex:badlink}
Let $\cA=(A_1,A_2,A_3)$ with
$$ A_1=\{(0,0),
(0,1),
(0,3),
(1,0),
(3,0)\}$$
$$A_2=\{
(0,0),
(0,1),
(0,3),
(1,0),
(3,0)\}$$
$$A_3=\{
(0,0),
(0,1),
(0,2),
(1,0),
(1,3),
(2,0),
(3,1),
(3,3)\}.$$
Choosing the specialization $S$ of every coefficient except the coefficient of the point $(0,0)$ in each configuration, we get that $\cT\cR_S(\cA)$ is a two-dimensional fan with f-vector $(1, 13, 17)$
living inside $\RR^3\subseteq\RR^{18}$. The link at $e_{11}\in\RR^{18}$
consists of $4$ rays.
The traversal of $\cT\cR_S(\cA)$ takes 79 seconds if Algorithm~\ref{alg:stablelinkprojections} is used but only 5 seconds if Algorithm~\ref{alg:stablelinkslicing} is used for computing the links.  Algorithm~\ref{alg:stablelinkprojections} needs to iterate through 2100 vertex pair choices at $e_{11}$, but much fewer for many of the other links.
\end{ex}


\subsection{Implicitization using specialized resultants}
\label{sec:imp}

In this section we will show that the tropicalization of a variety parameterized by polynomials with generic coefficients can be computed using specialized tropical resultants.  Let $f_1, f_2, \dots, f_k \in \CC[x_1^{\pm 1}, x_2^{\pm 1}, \dots, x_n^{\pm 1}]$ be polynomials parameterizing a variety $X$ in $\CC^k$.  Let $\Gamma$ be the {\em graph} of the parameterizing map, defined by $\langle y_1 - f_1, y_2 - f_2, \dots, y_k - f_k \rangle$ in $\CC[x_1^{\pm 1},x_2^{\pm 1},\dots,x_n^{\pm 1}, y_1, y_2, \dots, y_k]$.  When $f_1, f_2, \dots, f_k$ have generic coefficients, the tropical variety of $\Gamma$ is the stable intersection of the tropical hypersurfaces of the polynomials $ y_1 - f_1, y_2 - f_2, \dots, y_k - f_k$.  Since $X$ is the closure of the projection of $\Gamma \subset (\CC^*)^n \times \CC^k$ onto $\CC^k$, by tropical elimination theory, we can compute the tropical variety $\cT(X)$ as a projection of $\cT(\Gamma)$.  This approach was used in \cite{STY, SturmfelsYu08}.

Another way to compute $\cT(X)$ is by using specialized resultants.  Let $\cA = (A_1, A_2, \dots, A_k)$ where $A_i = \supp(f_i) \sqcup \{0\}$ for each $i = 1,2,\dots,k$.  Let $S = (\supp(f_1), \supp(f_2), \dots, \supp(f_k))$ be the sets of points to specialize, and let $V_S$ be the subspace of $\cprod_{i=1}^k \RR^{A_i} \times \RR^n$ defined by setting the coordinates in $S$ to $0$.

\begin{pro}
With the notation above, $\cT(X) = \cT\cR_S(\cA)$, i.e.\ the tropicalization of a variety parameterized by polynomials with generic coefficients coincides with a specialized resultant.
\end{pro}

\begin{proof}
Let $W$ be the incidence variety in $\cprod_{i=1}^k (\CC^*)^{A_i} \times (\CC^*)^n$ as in (\ref{eqn:incidenceVariety}), defined by equations of the form $y_i - g_i$ where $g_i$ is a polynomial with the same support as $f_i$ but with indeterminate coefficients.  Then the graph $\Gamma$ is obtained by specializing the coefficients of $g_i$ to those of $f_i$.  Since the coefficients of $f_i$ were assumed to be generic, we get $\cT(\Gamma) = \cT(W) \stint V_S$. 
By tropical elimination, $\cT(X) + (\{0\} \times \RR^n) = \cT(\Gamma) + ( \{0\} \times \RR^{n})$ in $\RR^k \times \RR^n$, which is in turn embedded in $\cprod_{i=1}^k \RR^{A_i} \times \RR^n$.  By the following lemma, $\cT(\Gamma) + ( \{0\} \times \RR^{n}) =  (\cT(W) + (\{0\} \times \RR^{n})) \stint V_S$.  After quotienting out both sides by $\{0\} \times \RR^n$, which is in the lineality space, we obtain $\cT(X) = \cT\cR_S(\cA)$.
\end{proof}

\begin{lem}
\label{lem:stableCommutes}
Let $\cF$ be a locally balanced fan in $\RR^N$.  Let $L$ and $L'$ be linear subspaces of $\RR^N$ such that $L' \subset L$.  Then
$$
(\cF \stint L) + L' = (\cF + L') \stint L
$$
In other words, stable intersection with a linear space commutes with Minkowski sum with a smaller linear space.
\end{lem}

\begin{proof}
Both $(\cF \stint L) + L'$ and $ (\cF + L') \stint L$ are empty if $\cF + L$ has dimension less than $N$.  Suppose this is not the case.  Then both sets contain $L'$ in their lineality space and consist of points of the form $u+v \in \RR^N$ where $u \in L'$ and $v \in \cF \cap L$ are such that $\dim(\link_v(\cF) + L) = N$.
\end{proof}

Since the tropical variety of the graph $\Gamma$ only depends on the extreme monomials of the parameterizing polynomials, the next result follows immediately.

\begin{cor}
\label{cor:impExtreme}
When using specialized resultants for implicitization, the extreme monomials of the input polynomials determine the tropical variety of the parameterized variety, so we can safely disregard the non-extreme terms.
\end{cor}

Using specialized resultants for implicitization instead of the approach in \cite{STY, SturmfelsYu08}  has the advantage that the computation of $\cT(\Gamma)$ as a stable intersection can be avoided.  Experiments show that the resultant description may speed up the reconstruction of the Newton polytope in some cases.  See Section~\ref{sec:comparison} for examples.

Moreover, when the variety $X$ is not a hypersurface, our resultant description gives a fan structure of $\cT(X)$ derived from the restriction of a secondary fan to a linear subspace, which is the normal fan of a fiber polytope.  Tropical elimination does {\em not} give a fan structure for varieties of codimension more than one.

\subsection{Tropical elimination for specialized tropical resultants}
\label{sec:elim}

As before, let $\cA$ be a tuple of point configurations in $\ZZ^n$ and $S$ be the tuple of subsets to be specialized.  
Let $W$ be the incidence variety and  $\cT W$ be is tropicalization as in Section~\ref{sec:simple}.  Let $W_S$ be a variety cut out by polynomials $f_i$ where the coefficients of monomials in $S$ have been specialized.  Then  $f_1, f_2, \dots, f_k$ may no longer form a tropical basis, but the tropicalization of $W_S$ can be computed as the stable intersection of tropical hypersurfaces of $f_1, f_2, \dots, f_k$ because the coefficients are assumed to be generic (or indeterminates).  The incidence variety $W$ is irreducible because it is a vector bundle over $(\CC^*)^n$, and although specializing coefficients may make $W_S$ reducible, all the irreducible components have the same tropical variety as seen in the paragraph above Proposition~\ref{prop:special}.  Hence any stable intersection of tropical hypersurfaces is connected in codimension~$1$, and we can use fan traversal to compute the stable intersection of hypersurfaces.

The specialized resultant is the projection of $W_S$ onto the non-specialized coefficient variables, and we can compute this using tropical elimination theory, which gives the tropical variety as a union of cones.  When the specialized tropical resultant is a tropical hypersurface, then we can reconstruct the normal fan of the dual Newton polytope using the methods in the next section.

The tropical hypersurface of $f_i$ only depends on the Newton polytope $P_i$ of $f_i$.  The non-specialized points in $A_i$ always contribute as vertices of $P_i$, but some specialized points of $A_i$ may not.  From this observation, we obtain the following result, which is not obvious from the resultant point of view.
\begin{lem}
\label{lem:interiorPoints}
If $a_{ij} \in A_i$ is a specialized point lying in the convex hull of other specialized points in $A_i$, then removing $a_{ij}$ from $A_i$ does not change the specialized tropical resultant.
\end{lem}

In other words, we may disregard the non-vertices among the specialized points because the Newton polytope and the tropical hypersurface of $f_i$ remain the same.  Using this lemma, we may be able to reduce the amount of work for computing specialized tropical resultants or specialized resultant polytopes.

\section{Polytope reconstruction}
\label{sec:polytopeReconstruction}

In this section we describe an algorithm for finding a fan structure
on a tropical hypersurface $T\subseteq \RR^n$.  Recall that the tropical hypersurface of a polytope $P \subset \RR^n$ is the set of $\omega \in \RR^n$ for which there exist distinct $p,q \in P$ such that for any $r\in P$, $\omega \cdot p = \omega \cdot q \leq \omega \cdot r$.  In other words, the tropical hypersurface of a polytope is the union of the normal cones to the polytope at the edges.  The multiplicity of a point in the relative interior of such a normal cone is the (lattice) length of the edge.
The tropical hypersurface of a polynomial is the tropical hypersurface of its Newton polytope.

 The tropical hypersurface $T$
will be presented to us as a finite collection of codimension 1 cones which may overlap badly but whose union is $T$. What we wish to compute is a
collection of codimension 1 cones such that the collection of all
their faces is a polyhedral fan with support $T$. This fan is
not unique unless we require it to be the coarsest --- that is, that it is the normal fan of the polytope defining $T$ with its maximal cones removed. 
If the codimension $1$ cones come with a multiplicity then an advantage of having the fan structure is that it is straightforward to reconstruct the 1-skeleton of the polytope defining $T$, hence the vertices of the polytope, up to translation. Therefore we will consider the computations of the vertices of a polytope, the normal fan, and the tropical hypersurface with the coarsest fan structure to be equivalent in what follows.

One way to perform the polytope reconstruction is to use the beneath-beyond method for computing convex hulls. The key observation is that for any generic $\omega\in\RR^n$ the vertex $\textup{face}_\omega(\textup{New}(f))$ can be computed
using ``ray shooting.'' See  \cite{DFS} and \cite{CTY}. 
The method we present in this paper uses the adjacency decomposition approach (see Section~\ref{sec:algorithms}) and the following algorithm for computing normal cones at vertices of the polytope defining $T$.
\begin{alg}[Region($S$,$\omega$)]
\label{alg:region}$ $\\
{\bf Input:} A collection $S$ of codimension 1 cones in $\RR^n$ such that $T:=\cup_{C\in S}C$ is the support of a tropical hypersurface. A vector $\omega\in\RR^n\setminus T$.\\
{\bf Output:} The (open) connected component of $\RR\setminus T$ containing $\omega$.
\begin{itemize}
\item $R:=\RR^n$.
\item For each $C\in S$: 
\begin{itemize}
\item While $R\cap C\not=\emptyset$:
\begin{itemize}
\item Find a point $p\in R\cap C$.
\item Introduce the parameter $\varepsilon>0$ and let $h$ be the open half line from $\omega$ through $p+\sum_{i=1}^n\varepsilon^i e_i$. 
\item The set of cones which intersect $h$ is the same for all $\varepsilon>0$ sufficiently small. Furthermore, the ordering of the intersection points along $h$ is fixed for $\varepsilon>0$ sufficiently small.
Among the cones that intersect $h$, let $D$ be a cone whose intersection point is closest to $\omega$.  (The choice of $D$ is not unique because the cones in $S$ need not form a fan and may overlap each other arbitrarily).
\item Let the halfspace $H\subset\RR^n$ be the connected component of $\RR^n\setminus\spann(D)$ containing $\omega$.
\item $R:=R\cap H$.
\end{itemize}
\end{itemize}
\item Return $R$.
\end{itemize}
\end{alg}
\begin{proof}
The set $R$ stays open and convex throughout the computation. At the end $R\cap T=\emptyset$. Each added constraint $H$ for $R$ is necessarily satisfied by the connected component because of its convexity. The symbolic perturbation of $p$ and the convexity of $R$ ensures that $H$ is independent of the choice of $D$, as all the possible choices of cones must be parallel. In fact, the set of constraints gives an irredundant inequality description of the returned cone.
\end{proof}

In computational geometry a standard way of handling the parameter $\varepsilon>0$ is to pass to the ordered field $\RR(\varepsilon)$. Since perturbed values are never multiplied together, there is no exponent growth. Indeed, the implementation is relatively simple.

\begin{pro}
\label{prop:numberOfChecks}
Let $a$ be the number of facets of the closure of the returned cone of Algorithm~\ref{alg:region}. The number of checks ``$R\cap C\not=\emptyset$'' performed in algorithm is $|S|+a$ while the number of interior point computations ``$p\in R\cap C$'' is $a$.
\end{pro}
\begin{proof}
The check is done for every cone in $C\in S$. In addition, whenever the algorithm enters the body of the while loop, a facet constraint $H$ is added to $R$, and an additional check ``$R\cap C\not=\emptyset$'' and a computation of $p$ is performed.
\end{proof}

The condition that the generic $h$ intersects a given polyhedral cone $C$ can be phrased as a condition on the ordering in which $h$ intersects the defining hyperplanes of $C$.  We can imagine moving a point starting from $\omega$ and along the half-line $h$, keeping track of which equations and inequalities defining $C$ are satisfied and updating when a defining hyperplane of $C$ is crossed.
Hence the implementation reduces to a check of the order in which $h$ intersects two given hyperplanes. The perturbation in such a check is not difficult to handle symbolically. The check can be used again to actually find a $D$ in the algorithm with intersection point closest to $\omega$.

To apply the adjacency decomposition approach we must be able to compute a starting cone and move across ridges to find neighboring cones, while computing links at ridges is trivial for complete fans. To find a starting cone we guess a vector outside $T$ and apply Algorithm~\ref{alg:region}. Suppose now that $C$ is a full dimensional cone in the normal fan and $u$ is a relative interior point on a facet of $C$ with outer normal vector $v$. For $\varepsilon>0$ sufficiently small, calling Algorithm~\ref{alg:region} with argument $u+\varepsilon v$ will give us the desired neighboring cone. In our implementation we again use comparison of intersection points on line segments to find an $\varepsilon$ sufficiently small to avoid all hyperplanes appearing in the description of $T$.

If we precompute generators for the cones in $S$ then most of the checks for empty intersection with $R$ can be done without using linear programming, but rather for each defining hyperplane of $R$ checking if the cone generators are completely contained on the wrong side. In our current implementation the time spent on finding first intersections along the half-lines is comparable to the time spent on linear programming. We present two examples to illustrate the usability of our algorithm. These examples appeared earlier in the literature.

\begin{ex}
\label{ex:hyperdeterminant}
The f-vector of the tropical hypersurface of the $2\times 2\times 2\times 2$ hyperdeterminant was computed in \cite{HSYY}. The support of the hypersurface is the sum of a tropical linear space and a classical linear space in $\RR^{16}$ and is easy to write as a union of cones. We reconstruct the 25448 normal cones of the Newton polytope of the defining equation in 163 minutes. Exploiting the 384 order symmetry as explained in \cite{traversingsymmetricfans} we reduce the running time to 7 minutes for computing the 111 orbits of maximal cones. With suitable input files the following Gfan command \cite{gfan} will compute the f-vector.  See also Section~\ref{sec:comparison} for further details.

\begin{footnotesize}
\begin{verbatim}
anders@gureko:~$ gfan_tropicalhypersurfacereconstruction -i troplinspc.fan
--sum --symmetry <claslinspc_and_symmetry.txt | grep -A1 F_VECTOR
F_VECTOR
1 268 5012 39680 176604 495936 927244 1176976 1005946 555280 178780 25448
\end{verbatim}
\end{footnotesize}
\end{ex}

\begin{ex}
\label{ex:implicitizationchallenge}
The implicitization challenge solved in \cite{CTY} is to reconstruct the Newton polytope of the defining equation of a tropical variety given as a union of $6 865 824$ cones. This $11$-dimensional polytope lives in $\RR^{16}$ and has a symmetry group of order $384$.  Its vertices come in 44938 orbits. In \cite{CTY}, a modified version of the ray-shooting method was used to produce coordinates of the vertices at a rate of a few (2-5) minutes per vertex.  Each round took about 45 minutes found 10-20 vertices typically.  However, a lot more computation, with some human interaction and parallelization, over a period of a few months was required to make sure that all the vertices were discovered, and this was done by computing the tangent cone at each found vertex, up to symmetry.  During the process most vertices were re-discovered multiple times.

On this example our new implementation in Gfan spends approximately 1 minute for each call of Algorithm~\ref{alg:region}. We estimate that the enumeration of the $44 938$ orbits would finish after 30 days of computation.  With the new method, we do not need to process a vertex more than once, and we obtain all the facet directions as the rays in the normal fan and all the tangent cones as duals of the normal cones.  Moreover, there is no post-processing needed to certify that all vertices have been found.
\end{ex}

The method we just described does not make use of multiplicities.  In fact, it is not necessary that the fan is polytopal, or even locally balanced.  We only require that each connected component of the complement of $T$ is convex.  

Before settling with Algorithm~\ref{alg:region} we also experimented with storing the codimension one cones in a \emph{binary space partitioning tree} (BSP tree). See \cite{Thibault} for a definition of BSP trees and an application to a computational geometry problem in arbitrary dimension. The tree would be built at initialization, and the connected components of the complement could be computed by gathering convex regions stored in the tree. This method worked as well as Algorithm~\ref{alg:region} in small dimensions, but in higher dimensions, like the examples above, Algorithm~\ref{alg:region} would always perform better. In Example~\ref{ex:hyperdeterminant} the difference would be a factor of five without exploiting symmetry. But in Example~\ref{ex:implicitizationchallenge} the number of required nodes of the tree would grow too large to have any chance of fitting in memory. The intuition behind the explosion in complexity is that cones (for example, simplicial cones of codimension one) in a higher dimensional space have larger chances of intersecting a fixed hyperplane.  Therefore in the process of building the BSP tree, a codimension one cone from the input will meet many other hyperplanes coming from other cones, causing an explosion in the number of nodes in the BSP tree.

\section{Comparison of algorithms}

\label{sec:comparison}

In this section, we consider various algorithms and compare the combinatorial complexity of the output (e.g.\ f-vector) and running time (recorded on a laptop computer with a 2.66 GHz Intel Core i5 processor and 8GB of memory).  All implementations are single threaded, done in C++ using cddlib \cite{cdd} and SoPlex \cite{wunderling}, and will be part of Gfan in its next release, unless otherwise noted. The combinatorial complexity of the output is essential for a fair comparison since different amounts of effort went into making each of the implementations fast. We mostly concentrated on the implementation of Algorithm~\ref{alg:region} and the secondary fan computation because of their broad range of applications, while less optimization effort has gone into algorithms specific to tropical resultants.

In general, the software Gfan uses the {\em max} convention for tropical varieties and Gr\"obner fans.  However, for the fact that the secondary fan of a point configuration is a coarsening of the Gr\"obner fan of the associated binomial (lattice) ideal, we need the subdivisions to be defined with respect to {\em min} if  the initial ideals are defined with respect to {\em max}.  Therefore Gfan uses {\em min} for secondary fans.  As tropical resultants are subfans of secondary fans, we chose to use {\em min} in this paper for tropical addition.

\subsection*{Hypersurfaces}

Let us first consider the case where the resultant variety $\cR(\cA)$ is a hypersurface.  Following is a list of different methods for computing the resultant polytope (or its tropical hypersurface or its normal fan).
\begin{enumerate}
\item Enumerating the vertices of the secondary polytope of $\Cay(\cA)$, and then using Sturmfels' formula  \cite[Theorem~2.1]{Sturmfels94} to obtain the vertices of the resultant polytope. We did not make an implementation but list only the time spent computing the secondary fan with the Gfan command

\begin{footnotesize}
\begin{verbatim}
gfan_secondaryfan <cayley.txt
\end{verbatim}
\end{footnotesize}

\item Computing the tropical hypersurface of the resultant as a subfan of the secondary fan by fan traversal using the methods described in Section~\ref{sec:algorithms}.

\begin{footnotesize}
\begin{verbatim}
gfan_resultantfan --vectorinput <tuple.txt
\end{verbatim}
\end{footnotesize}

\item Constructing the normal fan of the resultant polytope from the simple description of the tropical resultant as a union of cones as in Theorem~\ref{thm:orthants}. Our implementation in Gfan uses Algorithm~\ref{alg:region} for this.

\begin{footnotesize}
\begin{verbatim}
gfan_resultantfan --vectorinput --projection <tuple.txt
\end{verbatim}
\end{footnotesize}

\item Using Sturmfels' formula \cite[Theorem~2.1]{Sturmfels94} for finding the optimal vertex of the resultant polytope in a generic direction together with the beneath-beyond convex hull algorithm for recovering the whole polytope. The software ResPol \cite{EFKP} is a recent implementation of this method using the CGAL library.

\end{enumerate}

For the third approach, one can also use other methods for reconstructing a polytope from its tropical hypersurface, such as ray-shooting/beneath-beyond and BSP trees, as discussed in Section~\ref{sec:polytopeReconstruction}, although we found Algorithm~\ref{alg:region} to perform better, especially for polytopes of dimension 5 or more (compared to beneath-beyond in iB4e \cite{Huggins} and BSP).

For Example~\ref{ex:3tri} above, each of the first three methods finished in under one second in Gfan.  We present more challenging examples below. In the examples each matrix represents the point configuration consisting of its columns.

\smallskip
\noindent {\bf Example (a).} 
$$ \cA=\left(\bgroup\begin{pmatrix}0&1&3\\0&0&1\\1&1&1\\\end{pmatrix}\egroup,\bgroup\begin{pmatrix}0&0&1\\0&2&1\\3&2&0\\\end{pmatrix}\egroup,\bgroup\begin{pmatrix}0&2&2\\2&1&2\\3&1&1\\\end{pmatrix}\egroup,\bgroup\begin{pmatrix}1&2&2\\2&0&3\\1&0&2\\\end{pmatrix}\egroup\right)$$

\noindent
\begin{footnotesize}
\begin{tabular}{|l|l|r|}
\hline
Method/fan & F-vector of output & Timing \\
\hline
 (1) secondary fan  & 1 10432 55277 106216 88509 27140 & 467 s\\
 (2) traversing tropical resultant  & 1 5152 21406 28777 12614 & 733 s\\
(3) normal fan from simple description & 1 78 348 570 391 93  & 1.4 s\\
(4) beneath-beyond (ResPol) & 1 - - - - 93 & 2.7 s\\
\hline
\end{tabular}
\end{footnotesize}
\smallskip

\noindent {\bf Example (b).} 
 $$  \cA=\left(\bgroup\begin{pmatrix}0&0&1&3\\0&1&2&0\\\end{pmatrix}\egroup,\bgroup\begin{pmatrix}1&2&3&3\\1&2&0&1\\\end{pmatrix}\egroup,\bgroup\begin{pmatrix}0&1&2&3\\1&1&0&3\\\end{pmatrix}\egroup \right)$$


\noindent
\begin{footnotesize}
\begin{tabular}{|l|l|r|}
\hline
Method/fan & F-vector of output & Timing \\
\hline
 (1) secondary fan  & 1 3048 38348 178426 407991 494017 304433 75283 & 506 s\\
 (2) tropical resultant & 1 2324 26316 106083 197576 173689 58451 &1238 s \\ 
(3) normal fan&  1 56 497 1779 3191 3018 1412 249 & 6 s\\
(4) beneath-beyond (ResPol) & 1 - - - - - - 249 & 35 s\\
\hline
\end{tabular}
\end{footnotesize}

\smallskip
\noindent {\bf Example (c).}
$$ \cA=\left(\bgroup\begin{pmatrix}1&2&2\\1&2&3\\3&1&2\\1&2&2\\\end{pmatrix}\egroup,\bgroup\begin{pmatrix}1&3&3\\1&2&2\\0&1&3\\3&3&1\\\end{pmatrix}\egroup,\bgroup\begin{pmatrix}0&2&2\\2&0&2\\2&3&0\\1&3&0\\\end{pmatrix}\egroup,\bgroup\begin{pmatrix}1&1&3\\2&3&3\\0&1&0\\0&3&2\\\end{pmatrix}\egroup,\bgroup\begin{pmatrix}1&3&3\\3&2&2\\1&1&2\\3&0&2\\\end{pmatrix}\egroup\right) $$

\noindent
\begin{footnotesize}
\begin{tabular}{|l|l|r|}
\hline
Method/fan & F-vector of output & Timing \\
\hline
(3) normal fan from simple descr. &1 937 5257 11288 11572 5589 985   & 55 s  \\
(4) beneath-beyond (ResPol) & 1 - - - - - 985 & 236 s\\
\hline
\end{tabular}
\end{footnotesize}
\smallskip 

\noindent In Example (c) we were not able to compute the secondary fan and the resultant fan with the secondary fan structure due to integer overflow in intermediate polyhedral computations. Gfan has been designed to work well for Gr\"obner fans, where the degrees of the polynomials are never very large, since that would prevent us from computing a single Gr\"obner basis anyway (except for binomial ideals). In Example (c), a primitive normal vector of a codimension 1 cone of the normal fan of the resultant is $(-32,0,32,27,0,-27,25,-25,0,0,51,-51,-87,0,87)$,
showing that the resultant has degree at least 32+27+25+51+87=222. On such examples overflows typically arise when trying to convert an exactly computed rational generator of a ray to a primitive vector of 32-bit integers. Algorithm~\ref{alg:region} will show similar behavior on other examples, for example when converting ``$p\in R\cap C$'' to a vector of 32-bit integers. We intend to fix these implementation problems in the future.
\smallskip

\subsection*{Hypersurfaces with Specialization}

If the specialized resultant is a hypersurface, then we can compute its tropical variety using the following methods.
\begin{enumerate}
\item Compute $\cT\cR_S(\cA)$ as a subfan of the restriction of the secondary fan to a subspace $U_S$ by fan traversal using the algorithms in Section~\ref{sec:special}.

\begin{footnotesize}
\begin{verbatim}
gfan_resultantfan --vectorinput --special <tuple_and_spcvec.txt
\end{verbatim}
\end{footnotesize}

\item Compute the stable intersection $\cT\cR_S(\cA) = \cT\cR(\cA) \stint \{U_S\}$ as a union of cones, using the simple description from Theorem~\ref{thm:orthants} and the characterization of stable intersections from Lemma~\ref{lem:computingstableintersections}.  Then reconstruct the normal fan of the dual polytope using Algorithm~\ref{alg:region}.

\begin{footnotesize}
\begin{verbatim}
gfan_resultantfan --vectorinput --special --projection <tup_and_sv.txt
\end{verbatim}
\end{footnotesize}

\item Compute the specialized tropical resultant as a union of cones using stable intersection of hypersurfaces and tropical elimination theory as in Section~\ref{sec:elim} and reconstruct the normal fan of the dual polytope using Algorithm~\ref{alg:region}.
%
We combine the commands (see also \cite{SturmfelsYu08}):

\begin{footnotesize}
\begin{verbatim}
gfan_tropicalstartingcone --stable >startingcone.txt
gfan_tropicaltraverse --stable <startingcone.txt >stable.fan
gfan_tropicalhypersurfacereconstruction --sum -i stable.fan <lnspc.txt
\end{verbatim}
\end{footnotesize}

\item For a generic direction, Sturmfels' formula \cite[Theorem~2.1]{Sturmfels94} gives the optimal vertex of the resultant polytope in that direction, which can then be projected to get a point in the Newton polytope of the specialized resultant polynomial.  This can be combined with the beneath-beyond convex hull algorithm for recovering the whole polytope.  The software ResPol was used in the timings below.

\end{enumerate}
In \cite{EKP}, the authors proposed computing a {\em silhouette} or a projection of the secondary polytope.  This is dual to computing the restriction of the secondary fan to a subspace.  We provide the results and timings of this dual computation for comparison.  


In the following examples specialized points are shown in non-black color.

\smallskip
\noindent{\bf Example (d).}
 $$ \cA=\left(\bgroup\begin{pmatrix}{\grey 0}&0&1&{\grey 1}\\{\grey 0}&1&0&{\grey 1}\\\end{pmatrix}\egroup,\bgroup\begin{pmatrix}0&1&{\grey 1}&{\grey 2}\\1&0&{\grey 1}&{\grey 2}\\\end{pmatrix}\egroup,\bgroup\begin{pmatrix}{\grey 0}&{\grey 1}&1&2\\{\grey 0}&{\grey 1}&2&1\\\end{pmatrix}\egroup\right) $$


\noindent
\begin{footnotesize}
\begin{tabular}{|l|l|r|}
\hline
Method/fan & F-vector & Timing \\
\hline
Restriction of secondary fan &1 372 2514 5829 5661 1976&26 s\\
(1) traversing tropical resultant  &1 126 476 561 212&14 s\\
(2) normal fan from stable intersection&1 25 127 250 211 65 &0.7 s\\
(3) normal fan from tropical elimination&1 25 127 250 211 65&1.4 s\\
(4) beneath-beyond (ResPol) & 1 - - - - 65 & 0.5 s\\
\hline
\end{tabular}
\end{footnotesize}

\smallskip
\noindent {\bf Example (e).}
 $$ \cA=\left(\bgroup\begin{pmatrix}{\grey 0}&0&1&{\grey 1}\\{\grey 0}&1&0&{\grey 1}\\\end{pmatrix}\egroup,\bgroup\begin{pmatrix}0&1&{\grey 1}&{\grey 2}\\1&0&{\grey 1}&{\grey 2}\\\end{pmatrix}\egroup,\bgroup\begin{pmatrix}{\grey 0}&1&1&2\\{\grey 0}&1&2&1\\\end{pmatrix}\egroup\right) $$

\noindent
\begin{footnotesize}
\begin{tabular}{|l|l|r|}
\hline
Method/fan & F-vector & Timing \\
\hline
Restriction of secondary fan &1 709 6955 24354 39464 30226 8870&116 s\\
(1) traversing tropical resultant &1 469 3993 11296 12853 5040&320 s\\
(2) normal fan from stbl.\ inters.&1 29 209 597 792 485 110&1.3 s\\
(3) normal fan from trop.\ elim.&1 29 209 597 792 485 110&3.2 s\\
(4) beneath-beyond (ResPol) &1 - - - - - 110 & 2.3 s\\
\hline
\end{tabular}
\end{footnotesize}



\smallskip
\noindent {\bf Example (f).}
 $$\cA=\left(\bgroup\begin{pmatrix}{\grey 1}&1&2&3\\{\grey 2}&2&3&2\\{\grey 0}&2&1&2\\\end{pmatrix}\egroup,\bgroup\begin{pmatrix}{\grey 0}&0&1&1\\{\grey 1}&2&1&1\\{\grey 0}&2&1&3\\\end{pmatrix}\egroup,\bgroup\begin{pmatrix}{\grey 1}&1&2&3\\{\grey 1}&3&3&2\\{\grey 1}&1&0&1\\\end{pmatrix}\egroup,\bgroup\begin{pmatrix}{\grey 1}&1&3&3\\{\grey 0}&2&0&1\\{\grey 3}&2&1&1\\\end{pmatrix}\egroup\right) $$ 

\noindent
\begin{footnotesize}
\begin{tabular}{|l|l|r|}
\hline
Method & F-vector & Timing \\
\hline
(2)&1 1566 19510 98143 265202 424620 413455 238425 73741 9156&798 s\\
(3)&1 1566 19510 98143 265202 424620 413455 238425 73741 9156&974 s\\
\hline
\end{tabular}
\end{footnotesize}
\smallskip 

\noindent The current version of ResPol could not complete the computation for this example. Furthermore, we could not apply method (1) because of 32-bit integer overflows as explained in Example~(c).

\subsection*{Implicitization of hypersurfaces} Implicitization is a special case of the specialized resultants, and we compare the three methods as before.
\smallskip

\noindent{\bf Example (g).} (Implicitization of a bicubic surface \cite[Example 3.4]{EmirisKotsireas})
$$ 
 \cA=\left({\bgroup\begin{pmatrix}0&{\grey 0}&{\grey 0}&{\grey 0}&{\grey 1}&{\grey 2}&{\grey 3}\\
                      0&{\grey 1}&{\grey 2}&{\grey 3}&{\grey 0}&{\grey 0}&{\grey 0}\\\end{pmatrix}\egroup,
\bgroup\begin{pmatrix}0&{\grey 0}&{\grey 0}&{\grey 1}&{\grey 2}&{\grey 3}\\
                      0&{\grey 1}&{\grey 3}&{\grey 0}&{\grey 0}&{\grey 0}\\\end{pmatrix}\egroup,}\right.$$
$$\left.\bgroup\begin{pmatrix}0&{\grey 0}&{\grey 0}&{\grey 1}&{\grey 1}&{\grey 1}&{\grey 1}&{\grey 2}&{\grey 2}&{\grey 2}&{\grey 2}&{\grey 3}&{\grey 3}&{\grey 3}\\
                      0&{\grey 1}&{\grey 2}&{\grey 0}&{\grey 1}&{\grey 2}&{\grey 3}&{\grey 0}&{\grey 1}&{\grey 2}&{\grey 3}&{\grey 1}&{\grey 2}&{\grey 3}\\\end{pmatrix}
\egroup\right)$$ 

\noindent
\begin{footnotesize}
\begin{tabular}{|l|l|r|r|}
\hline
Method/fan & F-vector & Timing & No interior points \\
\hline
Restriction of secondary fan &1 26 66 42& 5 s & 2 s\\
(1) traversing tropical resultant & 1 13 17&16 s & 4 s\\
(2) normal fan from stable inters.&1 5 9 6 & 171 s & 9 s\\
(3) normal fan from tropical elim. &1 5 9 6& 0.4 s & 0.4 s\\
(4) beneath-beyond (ResPol) &1 5 9 6& $<$ 0.1 s & $<$ 0.1 s\\
\hline
\end{tabular}
\end{footnotesize}

\smallskip

\noindent As we saw in Corollary~\ref{cor:impExtreme}, removing the non-extreme monomials from the parameterizing polynomials does not change the resultant polytope, and in this example, this also does not change the restriction of the secondary fan.  However, doing so speeds up the computations, as seen on the right most column. 

\smallskip

\noindent {\bf Example (h).} (Implicitization of a hypersurface in four dimensions)
 $$\cA=\left(\bgroup\begin{pmatrix}0&{\grey 0}&{\grey 2}&{\grey 4}\\0&{\grey 2}&{\grey 4}&{\grey 1}\\0&{\grey 2}&{\grey 4}&{\grey 1}\\\end{pmatrix}\egroup,\bgroup\begin{pmatrix}0&{\grey 1}&{\grey 2}&{\grey 3}\\0&{\grey 2}&{\grey 2}&{\grey 0}\\0&{\grey 1}&{\grey 4}&{\grey 1}\\\end{pmatrix}\egroup,\bgroup\begin{pmatrix}0&{\grey 2}&{\grey 3}&{\grey 4}\\0&{\grey 4}&{\grey 0}&{\grey 1}\\0&{\grey 2}&{\grey 4}&{\grey 2}\\\end{pmatrix}\egroup,\bgroup\begin{pmatrix}0&{\grey 0}&{\grey 4}&{\grey 4}\\0&{\grey 2}&{\grey 2}&{\grey 3}\\0&{\grey 4}&{\grey 2}&{\grey 3}\\\end{pmatrix}\egroup\right) $$

\noindent
\begin{footnotesize}
\begin{tabular}{|l|l|r|}
\hline
Method/fan & F-vector & Timing \\
\hline
(1) traversing tropical resultant & 1 10665 24204 13660 & 2 h 10 m\\
(2) normal fan from stable intersection & 1 111 358 368 121& 9 s\\
(3) normal fan from tropical elimination &1 111 358 368 121& 2.6 s\\
(4) beneath-beyond (ResPol) &1 111 358 368 121& 1.5 s\\
\hline
\end{tabular}
\end{footnotesize}

\smallskip
\noindent For (3), computing the polytope from the tropical hypersurface using ray-shooting and beneath-beyond took 47 s in the TrIm implementation \cite{SturmfelsYu08} using the library iB4e \cite{Huggins} on a slightly slower machine.

\smallskip

\noindent {\bf Example (i).} (Implicitization of a hypersurface in five dimensions)

\begin{footnotesize}
 $$\cA = \left(\bgroup\begin{pmatrix}0&{\grey 1}&{\grey 3}&{\grey 4}\\0&{\grey 1}&{\grey 4}&{\grey 4}\\0&{\grey 2}&{\grey 2}&{\grey 4}\\0&{\grey 2}&{\grey 4}&{\grey 0}\\\end{pmatrix}\egroup,\bgroup\begin{pmatrix}0&{\grey 0}&{\grey 1}&{\grey 3}\\0&{\grey 0}&{\grey 2}&{\grey 3}\\0&{\grey 1}&{\grey 1}&{\grey 3}\\0&{\grey 1}&{\grey 2}&{\grey 3}\\\end{pmatrix}\egroup,\bgroup\begin{pmatrix}0&{\grey 0}&{\grey 2}&{\grey 3}\\0&{\grey 1}&{\grey 4}&{\grey 2}\\0&{\grey 1}&{\grey 1}&{\grey 1}\\0&{\grey 4}&{\grey 2}&{\grey 3}\\\end{pmatrix}\egroup,\bgroup\begin{pmatrix}0&{\grey 1}&{\grey 2}&{\grey 3}\\0&{\grey 1}&{\grey 4}&{\grey 2}\\0&{\grey 0}&{\grey 1}&{\grey 0}\\0&{\grey 1}&{\grey 3}&{\grey 3}\\\end{pmatrix}\egroup,\bgroup\begin{pmatrix}0&{\grey 0}&{\grey 2}&{\grey 4}\\0&{\grey 4}&{\grey 1}&{\grey 3}\\0&{\grey 3}&{\grey 4}&{\grey 3}\\0&{\grey 1}&{\grey 3}&{\grey 1}\\\end{pmatrix}\egroup\right) $$

\end{footnotesize}

\noindent
\begin{footnotesize}
\begin{tabular}{|l|l|r|}
\hline
Method/fan & F-vector & Timing \\
\hline
(2) normal fan from stable inters. & 1 5932 23850 35116 22289 5093&351 s\\
(3) normal fan from tropical elim. &1 5932 23850 35116 22289 5093&184 s\\
(4) beneath-beyond (ResPol) &1 5932 23850 35116 22289 5093& 898 s\\
\hline
\end{tabular}
\end{footnotesize}

\smallskip
\noindent For (3), timing includes 17 seconds for computing the specialized tropical incidence variety.  The normal fan reconstruction computation in TrIm with iB4e took 3375 seconds on a slightly slower machine. 

\subsection*{Non-hypersurfaces}

When $\cR(\cA)$ is not a hypersurface, the only method we know for computing $\cT\cR(\cA)$ with a fan structure without knowing the defining ideal is to traverse the secondary fan of $\Cay(\cA)$ and enumerating just the secondary cones whose RMS contains a fully mixed cell.  There are other descriptions of tropical resultants as a set, such as Theorem~\ref{thm:orthants}, but none gives a fan structure.

\noindent {\bf Example (j).}
 $$\cA = \left(\bgroup\begin{pmatrix}0&2&4\\4&1&1\\\end{pmatrix}\egroup,\bgroup\begin{pmatrix}3&5&5\\1&0&4\\\end{pmatrix}\egroup,\bgroup\begin{pmatrix}3&4&5\\1&5&2\\\end{pmatrix}\egroup,\bgroup\begin{pmatrix}0&1&2\\4&3&5\\\end{pmatrix}\egroup\right) $$ 

\noindent
\begin{footnotesize}
\begin{tabular}{|l|l|r|}
\hline
Method/fan & F-vector & Timing \\
\hline
Secondary fan &1 8876 72744 222108 322303 225040 60977&478 s\\
Traversing tropical result. & 1 968 4495 6523 3000 & 81 s\\
\hline
\end{tabular}
\end{footnotesize}

\smallskip
\noindent
We used, respectively, the commands:
\begin{footnotesize}
\begin{verbatim}
gfan_secondaryfan <cayley.txt
gfan_resultantfan --vectorinput <tuple.txt
\end{verbatim}
\end{footnotesize}

\subsection*{Non-hypersurfaces with Specialization}

The only method here is to traverse $\cT\cR_S(\cA)$ as a subfan of a restriction of the secondary fan using the algorithms in Section~\ref{sec:special}.

\noindent {\bf Example (k).}
 $$ \cA=\left(\bgroup\begin{pmatrix}{\grey 0}&{\grey 2}&4\\{\grey 4}&{\grey 1}&1\\\end{pmatrix}\egroup,\bgroup\begin{pmatrix}{\grey 3}&{\grey 5}&5\\{\grey 1}&{\grey 0}&4\\\end{pmatrix}\egroup,\bgroup\begin{pmatrix}{\grey 3}&{\grey 4}&5\\{\grey 1}&{\grey 5}&2\\\end{pmatrix}\egroup,\bgroup\begin{pmatrix}{\grey 0}&1&2\\{\grey 4}&3&5\\\end{pmatrix}\egroup\right) $$ 


\noindent
\begin{footnotesize}
\begin{tabular}{|l|l|r|}
\hline
Method/fan & F-vector & Timing \\
\hline
Restriction of secondary fan & 1 4257 23969 48507 42260 13467 & 256 s\\
Traversing spec.\ tropical result. & 1 310 831 533 & 81 s\\
\hline
\end{tabular}
\end{footnotesize}

\smallskip
\noindent
We used, respectively, the commands:
\begin{footnotesize}
\begin{verbatim}
gfan_secondaryfan --restrictingfan subspace.fan <cayley.txt
gfan_resultantfan --vectorinput --special <tup_and_sv.txt
\end{verbatim}
\end{footnotesize}

\subsection{Conclusion}

The new method of using adjacency decomposition with Algorithm~\ref{alg:region} for constructing the normal fan of a polytope from it tropical hypersurface works very well in practice.  Our implementation is much faster than any existing implementation of the beneath-beyond method with ray-shooting for polytope reconstruction, and we think the gap will widen even more in higher dimension since this new method scales well --- multi-linearly with respect to the number of cones in input and the number of vertices and edges of the output polytope, as shown in Proposition~\ref{prop:numberOfChecks}.

The normal fan reconstruction method can be used together with either the simple description of tropical resultants (Theorem~\ref{thm:orthants}) or tropical elimination (Section~\ref{sec:elim}) for computing resultant polytopes efficiently.  Traversing the (specialized) tropical resultant as a subfan of (a restriction of) the secondary fan of the Cayley configuration is combinatorially interesting but not computationally competitive.  

For implicitization, the beneath-beyond method from \cite{EFKP} works faster than any of our ``tropical'' methods when the output polytope is low dimensional, while our methods seem to have an advantage in higher dimension (5 or more).  However, the method of \cite{EFKP} may have an advantage when there are many specialized points in the input configurations, as the number of cones in the tropical description increases rapidly.  See the last problem in Section~\ref{sec:openProblems} below.

For resultant varieties of codimension higher than one, whether specialized or not, we only know of one method for computing the tropicalization as a fan, without knowing the defining polynomials, which is to traverse the secondary fan of the Cayley configuration or a restriction of it to a subspace.

\section{Open problems}
\label{sec:openProblems}

\begin{description}
\item[Combinatorial classification of resultant polytopes] 
For $1$-dimensional point configurations, the combinatorics of the resultant polytope only depend on the (partial) order of the (not necessarily distinct) points in each $A_i$ \cite{GKZ}, so a combinatorial classification is easy to obtain.  No such classification is known even for point configurations in $\ZZ^2$. A concrete problem is to 
{\bf  classify 4-dimensional resultant polytopes combinatorially.} 
This was done for 3-dimensional resultant polytopes by Sturmfels \cite{Sturmfels94}, and only one-dimensional point configurations were needed for this case. To understand the 4-dimensional resultant polytopes, we need to work with the case $\cA = (A_1, A_2, A_3)$ where each $A_i$ consists of three points in $\ZZ^2$ that are not necessarily distinct.  How can we stratify the space of tuples $\cA$'s according to the combinatorial type of the resultant polytope? 
\item[Using symbolic perturbation]  At the end of Section~\ref{sec:codimension}, we gave a probabilistic algorithm for computing codimension of resultants. Can we turn this into a polynomial time deterministic algorithm using symbolic perturbation? 
\item[Finding a point in the specialized tropical resultant] For non-specialized tropical resultants, the polynomial time algorithm for computing codimension from Section~\ref{sec:codimension} can also be used to find a generic point, by Theorem~\ref{thm:orthants}. Is there a polynomial time algorithm for finding a generic vector $\omega\in\QQ(\varepsilon)^m$ in the specialized tropical resultant?
\item[Improved description of specialized tropical resultants] 
By combining the descriptions of tropical resultants in Theorem~\ref{thm:orthants} and stable intersections in Lemma~\ref{lem:computingstableintersections}, we get a specialized tropical resultant as a union of cones. In computations, we need to go through a list of $\prod_{i=1}^k{m_i \choose 2}$ choices of tuples of pairs from $A_i$, many of which do not contribute to a facet of specialized tropical resultant.
{\bf Give a combinatorial characterization for the choices of the tuples of pairs that contribute to a facet.} Corollary~\ref{cor:impExtreme} and Lemma~\ref{lem:interiorPoints} are results in this direction.
\end{description}

\section*{Acknowledgments}  We thank MSRI (Berkeley, USA) and Institut Mittag-Leffler (Djursholm, Sweden) for their support and hospitality.  The first author was partially supported by the German Research Foundation (Deutsche Forschungsgemeinschaft (DFG)) through the Institutional Strategy of the University of G\"ottingen, partially by the DFG grant MA 4797/3-1 (SPP 1489) and partially by the Danish Council for Independent Research, Natural Sciences (FNU). The second author was partially supported by a Postdoctoral Research Fellowship and DMS grant \#1101289 from the National Science Foundation (USA).  We would like to express our gratitude to the anonymous referees, especially Reviewer \#1, for exceptionally careful reading and detailed comments.

\bibliographystyle{amsalpha}

\bibliography{mybib}

\end{document}